\documentclass[a4paper,11pt]{article}
\usepackage[utf8]{inputenc}

\usepackage{amsmath}
\usepackage{amsfonts}
\usepackage{amssymb}
\usepackage{graphics}
\usepackage{graphicx}
\usepackage{epstopdf}
\usepackage{amsthm}
\usepackage{color}
\usepackage{multirow}
\usepackage{booktabs}
\usepackage{bm}
\usepackage{mathrsfs}
\usepackage[utf8]{inputenc}
\usepackage{url}
\usepackage{rotating}
\usepackage{ulem}

\newtheorem{proposition}{Proposition}

\newtheorem{remark}{Remark}

\newcommand{\R}{\mathbb{R}}

\title{Extreme and statistical properties of eigenvalue indices of simple connected  graphs}

\author{So\v{n}a Pavl\'{\i}kov\'a
\thanks{Alexander Dub\v{c}ek University of Tren\v{c}\'{\i}n, Slovakia}
\and
Daniel \v{S}ev\v{c}ovi\v{c} \thanks{Comenius University in Bratislava, Slovakia}
\and
Jozef \v{S}ir\'a\v{n} \thanks{Slovak Technical University in Bratislava, Slovakia}
}

\begin{document}

\maketitle

\begin{abstract}

We analyze graphs attaining the extreme values of various spectral indices in the class of all simple connected graphs, as well as in the class of graphs which are not complete multipartite graphs. We also present results on density of spectral gap indices and its nonpersistency with respect to small perturbations of the underlying graph. We show that a small change in the set set of edges may result in a significant change of the spectral index like, e.g., the spectral gap or spectral index. We also present a statistical and numerical analysis of spectral indices of graphs of the order $m\le 10$. We analyze the extreme values for spectral indices for graphs and their small perturbations. Finally, we present the statistical and extreme properties of graphs on $m\le 10$ vertices. 



\end{abstract}

\section{Introduction}
In theoretical chemistry, biology, or statistics, spectral indices and properties of graphs representing the structure of chemical molecules or transition diagrams for finite Markov chains play an important role (cf. Cvetkovi\'c \cite{Cvetkovic1988,Cvetkovic2004}, Brouwer and Haemers \cite{Brouwer2012} and references therein). In the past decades, various graph energies and indices have been proposed and analyzed. For example, the sum of absolute values of eigenvalues is called the matching energy index (cf. Chen and Liu \cite{Lin2016}), the maximum of the absolute values of the least positive and largest negative eigenvalue is related to the HOMO-LUMO index (see Mohar \cite{Mohar2013,Mohar2015}, Li {et al.} \cite{Li2013}, Jakli\'c  {et al.} \cite{Jaklic2012}, Fowler {et al.} \cite{Fowler2010}), their difference is related to the HOMO-LUMO separation gap (cf.  Gutman and Rouvray \cite{Gutman1979}, Li {et al.} \cite{Li2013}, Zhang and An \cite{Zhang2002}, Fowler  {et al.} \cite{Fowler2001}). 

The spectrum $\sigma(G_A)\equiv\sigma(A)$ of a simple nonoriented connected graph $G_A$ on $m$ vertices is given by the eigenvalues of its adjacency matrix $A$:
\[
\lambda_{max}\equiv\lambda_1 \ge \lambda_2 \ge \dots \ge \lambda_m\equiv\lambda_{min}.
\]
For a simple graph (without loops and multiple edges) we have $A_{ii}=0$, and so $\sum_{i=1}^m\lambda_i = trace(A) =0$. Hence $\lambda_1>0, \lambda_m<0$. 

In what follows, we shall denote $\lambda_+(A)$, and $\lambda_-(A)$ the least positive and largest negative eigenvalues of a symmetric matrix $A$ having positive and negative real eigenvalues. 
Let us denote by $\Lambda^{gap}(A)= \lambda_+(A) -  \lambda_-(A)$ and $\Lambda^{ind}(A)= \max(|\lambda_+(A)|, |\lambda_-(A)|)$ the spectral gap and the spectral index of a symmetric matrix $A$. Furthermore, we define the spectral power $\Lambda^{pow}(A) = \sum_{k=1}^m |\lambda_k|$. Clearly, all three  indices $\Lambda^{gap}, \Lambda^{ind}$, and $\Lambda^{pow}$ depend on positive $\sigma_+(A)=\{\lambda \in \sigma(A), \lambda>0\}$, and negative $\sigma_-(A)=\{\lambda \in \sigma(A), \lambda<0\}$ parts of the spectrum of the matrix $A$. In fact, $\lambda_+(A) = \min\sigma_+(A),\  \lambda_-(A) = \max\sigma_-(A)$, and $\Lambda^{pow}(A) = \sum_{\lambda\in\sigma_+(A)} \lambda - \sum_{\lambda\in\sigma_-(A)} \lambda$.

In the past decades, various concepts of introducing inverses of graphs based on inversion of the adjacency matrix have been proposed. 
In general, the inverse of the adjacency matrix does not need to define a graph again because it may contain negative elements (cf. \cite{Har}).
Godsil \cite{Godsil1985} proposed a successful approach to overcome this difficulty, which defined a graph to be (positively) invertible if the inverse of its nonsingular adjacency matrix is diagonally similar (cf. \cite{Zas}) to a nonnegative integral matrix representing the adjacency matrix of the inverse graph in which positive labels determine edge multiplicities. In the papers \cite {Pavlikova2016, Pavlikova2022-LAA},  Pavl\'\i kov\'a and \v{S}ev\v{c}ovi\v{c} extended this notion to a wider class of graphs by introducing the concept of negative invertibility of a graph. 

In chemical applications, the spectral gap $\Lambda^{gap}$ of a structural graph of a molecule is related to the so-called HOMO-LUMO energy separation gap of the energy of the highest occupied molecular orbital (HOMO) and the lowest unoccupied molecular orbital (LUMO).  Following Hückel’s molecular orbital method \cite{Huckel1931}, eigenvalues of a graph that describes an organic molecule are related to the energies of molecular orbitals (see also Streitwieser \cite[Chapter 5.1]{Streitwieser1961}). 
Finally, according Aihara \cite{Aihara1999JCP, Aihara1999TCH},  it is energetically unfavorable to add electrons to a high-lying LUMO orbital. Hence, a larger HOMO-LUMO gap implies a higher kinetic stability and low chemical reactivity of a molecule. Furthermore, the HOMO-LUMO energy gap generally decreases with the number of vertices in the structural graph (cf. \cite{Bacalis2009}). 

In this paper, we analyze the extreme and statistical properties of the spectrum of all simple connected graphs. It includes the analysis of maximal and minimal eigenvalues, as well as  indices such as, e.g., spectral gap, spectral index, and the power of spectrum. We analyze graphs that attain extreme values of various  indices in the class of all simple connected graphs, as well as in the class of graphs that are not complete multipartite graphs. We also present results on the density of spectral gap indices and its nonpersistency with respect to small perturbations of the underlying graph. We show that a small change in the set set of edges may result in a significant change of the spectral gap or spectral index. We also present a statistical and numerical analysis of  indices of graphs of order $m\le 10$. 

The paper is organized as follows. In Section 2 we first recall the known results on extreme values of maximal and minimal eigenvalues of adjacency matrices. We also report the number of all simple connected graphs due to McKay \cite{McKay}. Next, we analyze the extreme values for  indices for completed multipartite graphs and their small perturbations. In Section 3 we focus our attention on the statistical and extreme properties of graphs on $m\le 10$ vertices.

\section{Extreme properties of  indices}

\begin{figure}
    \centering
    \includegraphics[width=.35\textwidth]{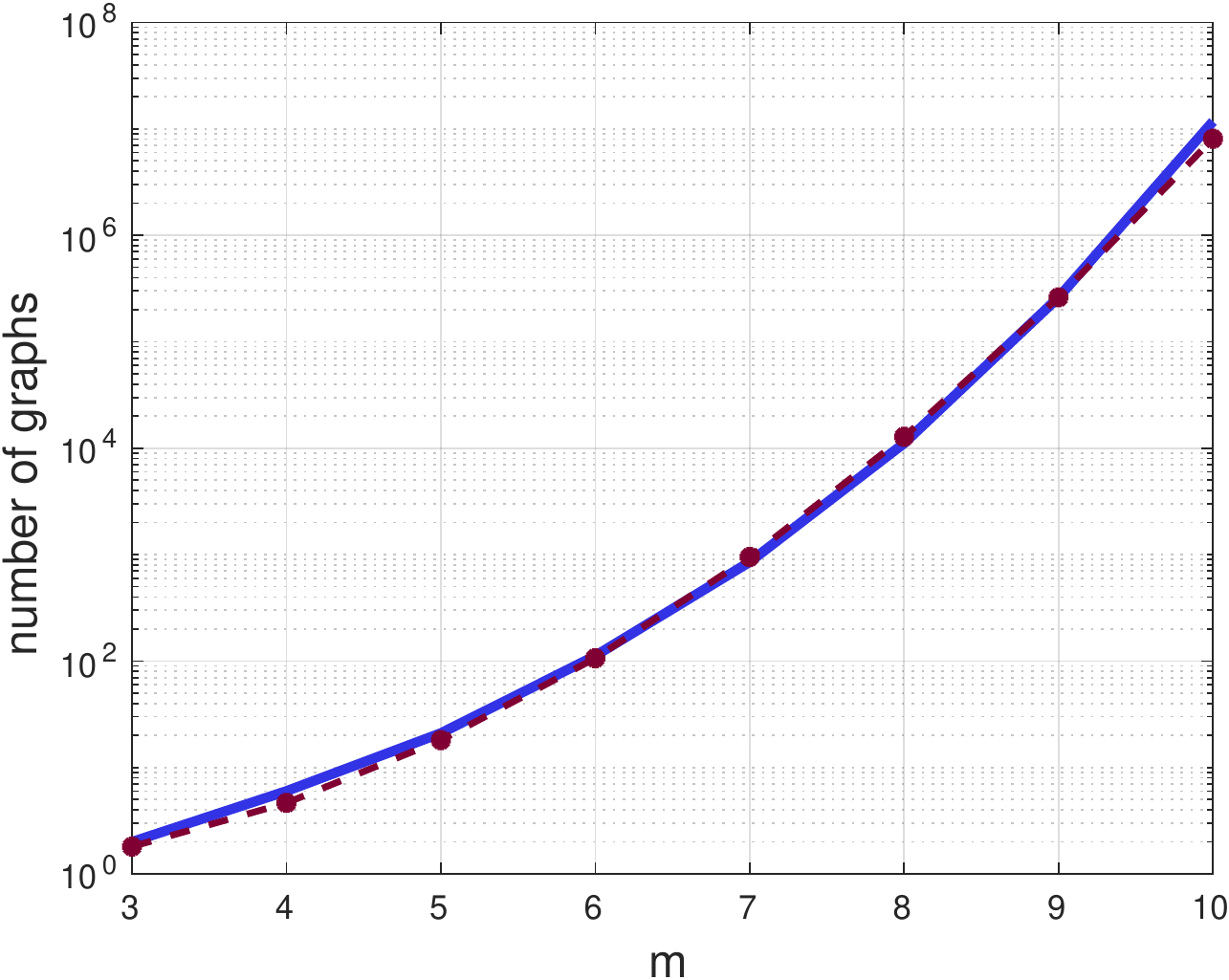} 
    
    \caption{\small The numbers $c_m$ of all simple connected as a function of number of vertices (blue solid line), and its approximation by means of the approximation formula (\ref{approxformula}) (red dashed line).} 
    \label{fig-graf-numbers-approximate}
\end{figure}

Denote by $c_m$ the number of simple non-isomorphic connected graphs on $m$ vertices. According to the McKay's list of all simple connected graphs \cite{McKay} the numbers $c_m, m\le 10$, are summarized in Table~\ref{tab-number}.

\begin{table}
\caption{\small Numbers of all simple connected graphs on $m\le 10$ vertices.}
\small
\begin{center}
\hglue -0.5truecm \begin{tabular}{l||l|l|l|l|l|l|l|l|l}
\hline
$m$  & $2$ & $3$ & $4$ & $5$ & $6$ & $7$ & $8$ & $9$ & $10$\\
\hline\hline
total \# & $1$   & $2$   & $6$   &  $21$ & $112$ & $853$ &$11117$&$261080$& $11716571$ 
\\
\hline
\end{tabular}
\end{center}
\label{tab-number}
\end{table}
Although there exists an approximation formula for the number of labelled simple connected graphs of the given order $m$ and number of edges (cf. Bender, Canfield, and McKay \cite{Bender1990}) for small values of $m$ the number $c_m$ can be approximated by the following compact the quadratic exponential function:
\begin{equation}
c_m\approx \omega_0 10^{\omega_1 (m-9)+\omega_2 (m-9)^2}, \quad \text{where} \ \omega_0=261080, \ \ \omega_1=1.4, \ \ \omega_2=0.09.
\label{approxformula}
\end{equation}
This formula is exact for $m=9$ and gives good approximation results for other orders $m\le10$ (see Fig.~\ref{fig-graf-numbers-approximate}).

Recall the following well-known facts: the maximal value of $\lambda_{max}=\lambda_1$ over all simple connected graphs on the $m$ vertices is equal to $m-1$, and it is attained by the complete graph $K_m$. The minimal value of $\lambda_{max}$ is equal to $2\cos(\pi/(m+1))$, and it is attained for the path graph $P_m$. Furthermore, the lower bound for the minimal eigenvalue $\lambda_{min}=\lambda_m \ge -\sqrt{\lfloor m/2\rfloor\lceil m/2\rceil}$ was independently proved in \cite{Con,Hon,Pow}. The lower bound is attained for the complete bipartite graph $K_{m_1,m_2}$ where $m_1=\lceil m/2\rceil, m_2=\lfloor m/2\rfloor$. The maximum value of $\lambda_{min}$ on all simple connected graphs on the $m$ vertices is equal to $-1$, and it is attained for the complete graph $K_m$.

\subsection{ Indices for complete multipartite graphs and their perturbations}

The aim of this section is to analyze  indices and their extreme values for simple connected graphs on the $m$ vertices. 

\begin{proposition}
\label{spectrum-completemultipartitioned}
Let us denote $K_{m_1, \dots, m_k}$ the complete multipartite graph where $1\le m_1\le \dots \le m_k$ denote the sizes of parts, $m_1+ \dots + m_k =m$, and $k\ge2$ is the number of parts. Then the spectrum of the adjacency matrix $A$ of $K_{m_1, \dots, m_k}$ satisfies $\sigma(A) \subseteq [-m_k, m - m/k]$. If $m_i<m_{i+1}$ then there exists a single eigenvalue $\lambda\in (-m_{i+1}, -m_i)$. If $m_i=m_{i+1}= \dots = m_{i+j}$ then $\lambda=-m_i$ is an eigenvalue of $A$ with multiplicity $j$.

Finally, $0<\lambda_+(A)\le m-m/k$ and $-m/k\le \lambda_-(A) <0$. As a consequence, $\Lambda^{gap}(A)\le m, \Lambda^{ind}(A)\le m-1$,  $\Lambda^{pow}(A)\le 2(m-m/k)$. The equalities for the indices $\Lambda^{gap}(A),  \Lambda^{ind}(A)$ are reached by the complete graph $G_A=K_m$.
\end{proposition}

\begin{proof}
The adjacency matrix $A$ of $K_{m_1, \dots, m_k}$ has the block form:
\[
A={\bf 1} {\bf 1}^T - diag(D_1, \dots, D_k),
\]
where ${\bf 1}= (1,\dots, 1)^T\in\mathbb{R}^m$, and $D_i$ is the $m_i\times m_i$ matrix consisting of ones. Now, if $\lambda$ is a nonzero eigenvalue of $A$ with an eigenvector $x=(x_1,\dots,x_m)^T$ then
\begin{equation}
\alpha - \alpha_p = \lambda x_l, \quad \text{for each}\ \ l=\mu_{p-1} + 1, \dots ,\mu_p, \ \quad \alpha_p = \sum_{l=1+\mu_{p-1} }^{ \mu_p } x_l, \quad \mu_p=\sum_{r=1}^{p} m_r, 
\label{alpha}
\end{equation}
for $p=1, \dots, k$. Here $\alpha=\sum_{p=1}^k \alpha_p = \sum_{j=1}^m x_j$. For example, if $p=1$ then $\sum_{j=1}^{m_{1}} x_j = \alpha m_1/(\lambda+m_1)$ provided that $\lambda\not=-m_1$. Similarly, we can proceed with the remaining parts $m_2, \dots, m_k$. In the case $\alpha=0$ we have $\lambda\in \{-m_1, \dots, -m_k\}$. In the case $\alpha\not=0$ we conclude $\lambda\not\in\{-m_1, \dots, -m_k\}$, and the eigenvalue $\lambda$ satisfies the rational equation:
\begin{equation}
\psi(\lambda)=1, \quad \text{where}\ \psi(\lambda)= \sum_{i=1}^k \frac{m_i}{\lambda+m_i}. 
\label{psieq}
\end{equation}
Conversely, if $\lambda\not\in\{-m_1, \dots, -m_k\}$ satisfies $\psi(\lambda)=1$ then it is easy to verify that the nontrivial vector $x\in\mathbb{R}^m$,
\[
x=( \underbrace{y_1, \dots ,y_1}_{m_1 \ \text{times}},
\underbrace{y_2, \dots ,y_2}_{m_2 \ \text{times}},
\dots,
\underbrace{y_k, \dots ,y_k}_{m_k \ \text{times}} )^T,
\quad\text{where}\ y_i=\frac{m_i}{\lambda+m_i}, 
\]
is an eigenvector of $A$, i.e. $A x =\lambda x$. 

In what follows, we shall derive necessary bounds on eigenvalues of $A$. Suppose to the contrary that $\lambda< -m_k$ is an eigenvalue of $A$. Then $\lambda+m_i\le \lambda+m_k <0$ for any $i=1,\dots,k$, and so $\psi(\lambda)<0<1$. Therefore, $\lambda\ge -m_k$ for any eigenvalue $\lambda\in\sigma(A)$. To derive an upper bound for the positive eigenvalue of $A$ we introduce an auxiliary function $\phi(\xi_1,\dots, \xi_k) = \sum_{i=1}^k \frac{\xi_i}{\lambda+\xi_i}$ where $\lambda>0$ is fixed. The function $\phi:\mathbb{R}^k\to\mathbb{R}$ is concave. Using the Lagrange function $\mathscr{L}(\xi,\mu)= \phi(\xi_1,\dots, \xi_k) - \mu \sum_{i=1}^k \xi_i$ it is easy to verify that $\phi$ achieves the unique constrained maximum in the set $\{\xi\in\mathbb{R}^k, \sum_{i=1}^k \xi_i = m\}$ at the point $\hat\xi=(m/k, \dots,m/k)^T$. Therefore, for any $\lambda>0$ we have
\[
\psi(\lambda)=\sum_{i=1}^k \frac{m_i}{\lambda+m_i} = \phi(m_1, \dots, m_k) \le \phi(m/k, \dots,m/k) = \frac{m}{\lambda+m/k}.
\]
If $\lambda>0$ is a positive eigenvalue of $A$ then $\psi(\lambda)=1$ and so $\lambda+m/k\le m$, that is, $0<\lambda\le m - m/k$. Therefore, $\sigma(A) \subset [-m_k, m - m/k]$.

In the trivial case of an equipartite graph $K_{m_1, \dots, m_k}$ with $m_1=\dots = m_k =m/k$ we obtain $\lambda_-(A)\ge -m_k=-m/k$ and $\lambda_+(A)\le m-m/k$. Thus, $\Lambda^{gap}\le m$, and $\Lambda^{ind}\le m-m/k\le m-1$.  This estimate also follows from the results of \cite{EsHa} and \cite{Del}. Therefore, for any $1\le l<k$ we conclude that $\Lambda^{gap}(A)=\lambda_+(A)-\lambda_-(A)\le m-m/k - (-m/k) =m$. Similarly,  $\Lambda^{ind}(A)\le m-1$.

Now, consider a non-equipartite graph $K_{m_1,\dots, m_k}$ with $m_1= \dots =m_l < m_{l+1} \le \dots \le m_k$ where $1\le l<k$. Suppose that $l=1$, that is, $1\le m_1 < m_2\le \dots \le m_k$. The function $\psi$ is strictly decreasing in the interval $(-m_2, -m_1)$ with infinite limits $\pm\infty$ when $\lambda\to -m_2$ and $\lambda\to -m_1$, respectively. Therefore, there exists a unique eigenvalue $\lambda\in (-m_2, -m_1)$ of the matrix $A$. We have $m_1 + (k-1) m_2 \le \sum_{i=1}^k m_i = m$. Define $\tilde\lambda = - m_1/k - m_2 (k-1)/k$. Then $\tilde\lambda \ge -m/k$. In what follows we shall prove that $\psi(\tilde\lambda)\ge 1$. The function $\xi\mapsto \xi/(\tilde\lambda +\xi)$ decreases for $\xi>-\tilde\lambda$. Therefore
\begin{eqnarray*}
\psi(\tilde\lambda) &\ge& \frac{m_1}{\tilde\lambda+m_1} + (k-1) \frac{m_2}{\tilde\lambda+m_2}
= -\frac{k}{k-1} \frac{m_1}{m_2 - m_1} + k(k-1) \frac{m_1}{m_2 - m_1}
\\
&=& \frac{k}{k-1} \frac{(k-1)^2 m_2 - m_1}{m_2 - m_1} \ge \frac{k}{k-1}>1,
\end{eqnarray*}
because $k\ge 2$. Since $\psi$ is strictly decreasing in the interval $(-m_2, -m_1)$ we have $-m/k\le \tilde\lambda < \lambda$ because $\psi(\lambda)=1$.

In the case $l\ge 2$ we can apply a simple perturbation argument. Indeed, let us perturb the adjacency matrix $A$ by a small parameter $0<\varepsilon\ll 1$ as follows: 
\[
A^\varepsilon = {\bf 1} {\bf 1}^T - diag((1-\varepsilon) D_1, D_2, \dots, D_{l-1}, (1+\varepsilon)D_l, D_{l+1}, \dots, D_k).
\]
It corresponds to the perturbation $m_1^\varepsilon = (1-\varepsilon)m_1, m_l^\varepsilon = (1+\varepsilon)m_l$. All remaining $m_i$ remain unchanged for $i\not=1$ and $i\not=l$. Then for the corresponding perturbed function $\psi^\varepsilon$ there exists a solution $\lambda^\varepsilon \in (m_1-\varepsilon, m_1)$ of the equation $\psi^\varepsilon(\lambda^\varepsilon)=1$. Since the spectrum of $A^\varepsilon$ depends continuously on the parameter $\varepsilon\to 0$, we see that $\lambda^\varepsilon\to\lambda=-m_1 = \dots = -m_l$ is an eigenvalue of the graph $G_A$ provided that $l\ge 2$. In this case $\lambda=-m_1\ge - m/k$.

A complete multipartite graph $G_A=K_{m_1,m_2, \dots, m_k}$ has exactly one positive eigenvalue $\lambda_1>0$ (cf. Smith \cite{CvDS}). Since $\sum_{i=1}^m\lambda_i=0$ we have $\Lambda^{pow}(A)=\sum_{i=1}^m|\lambda_i|= 2\lambda_1\le 2(m-m/k)$.
The spectrum of the complete graph $K_m$ consists of eigenvalues $m-1$, and $-1$ with multiplicity $m-1$. Therefore, $\Lambda^{gap}=m, \Lambda^{ind}=m-1$, as claimed.
\end{proof}

\begin{remark}
The main idea of the proof of Proposition~\ref{spectrum-completemultipartitioned} is a non-trivial generalization of the interlacing theorem \cite[Theorem 1]{EsHa} due to Esser and Harary. It is based on a solution $\lambda$ to the dispersion equation (\ref{psieq}), that is $\psi(\lambda)=1$ (see  \cite[Eq. (9)]{EsHa}). 
In \cite[Corollary 1]{EsHa} they showed that $\sigma(A) \subseteq [-m_k, m - m_1]$. Because $k m_1 \le \sum_{i=1}^k m_i =m$, we obtain $m-m/k \le m-m_1$. 
Using the concavity of the function $\phi:\mathbb{R}^k\to \mathbb{R}$ and the constrained optimization argument, we were able to improve this estimate. We derived the estimate $\sigma(A) \subseteq [-m_k, m - m/k]$ which yields optimal bounds $\Lambda^{gap}\le m, \Lambda^{ind}\le m-1$ derived in Proposition~\ref{spectrum-completemultipartitioned}. Furthermore, we introduced a novel analytic perturbation technique to handle the case when the sizes $m_1=\dots=m_l$ of parts coincide. 
\end{remark}

\begin{remark}
\label{Km1m2}
It follows from the proof of Proposition~\ref{spectrum-completemultipartitioned}  that $\lambda$ is an eigenvalue of $A$ if and only if the vector $z=(\alpha_1, \dots, \alpha_k)^T\in\mathbb{R}^k$ (see (\ref{alpha})) is an eigenvector of the $k\times k$ matrix $\mathscr{A}$, i.e. $\mathscr{A}z =\lambda z$, where $\mathscr{A}_{ij}= m_i$ for $i\not=j$, $\mathscr{A}_{ii}=0$. 

As a consequence, the spectrum of the complete bipartite graph $K_{m_1, m_2}$ consists of $m_1+m_2 -2$ zeros and $\pm\sqrt{m_1 m_2}$. Therefore, $\Lambda^{gap}(K_{m_1, m_2})=\Lambda^{pow}(K_{m_1, m_2}) = 2\sqrt{m_1 m_2}$, and $\Lambda^{ind}(K_{m_1, m_2}) = \sqrt{m_1 m_2}$.  Furthermore, if $m$ is even, then $\Lambda^{gap}(K_{m/2, m/2})=m=\Lambda^{gap}(K_{m})$, i.e., the complete bipartite graph $K_{m/2, m/2}$ as well as the complete graph $K_m$ maximize the spectral gap $\Lambda^{gap}$. 
The smallest example is the complete graph $K_4$ with eigenvalues $\{3, -1, -1,-1\}$ and the circle $C_4\equiv K_{2,2}$ with eigenvalues $\{2,0,0,-2\}$ that yields the same maximum value of $\Lambda^{gap}=4$.

Similarly, one can derive the equation for spectrum of the complete tripartite graph $K_{m_1, m_2, m_3}$. It leads to the following depressed cubic equation $\lambda^3 + r \lambda + s = 0$ with $r=-(m_1 m_2 + m_2 m_3 + m_1 m_3), s=-2 m_1 m_2 m_3$. However, the discriminant $\Delta= -(4r^3 + 27 s^2)$ is positive for a non-equipartite graph, and there are three real roots of the depressed cubic. With regard to Galois theory, roots cannot be expressed by an algebraic expression, and Cardano's formula leads to "{casus irreducibilis}". 

\end{remark}

\begin{proposition}
\label{stat-allgraphs}
Let us consider the class of all simple connected graphs on $m$ vertices. The following statements regarding the  indices $\Lambda^{gap}, \Lambda^{ind}$ and $\Lambda^{pow}$ hold.
\begin{itemize}
\item[a)]  If $G_A$ is not a complete multipartite graph of order $m$, then
$\Lambda^{gap}(A)\le m-1, \Lambda^{ind}(A)\le m/2$ for $m$ even, and $\Lambda^{gap}(A)\le m-3/2, \Lambda^{ind}(A)\le \sqrt{m^2-1}/2$ for $m$ odd.

\item[b)] The maximum value of $\Lambda^{pow}$ on the $m\le 7$ vertices is equal to $2m-2$, and it is attained for the complete graph $K_m$. For $m=7$ there are two maximizing graphs with $\Lambda^{pow}=12$ - the complete graph $K_7$ and the noncomplete graph shown in Fig.~\ref{fig-graf-maxLambdaPower-2-second}. Starting $m\ge 8$ the maximal $\Lambda^{pow}$ is attained by noncomplete graphs depicted in Fig.~\ref{fig-graf-maxLambdaPower-8-9-10} for $8\le m\le10$.  
\end{itemize}
\end{proposition}

\begin{proof}
According to Smith \cite{CvDS}, a simple connected graph has exactly one positive eigenvalue (i.e. $\lambda_2(A)\le 0$) if and only if it is a complete multipartite graph $K_{m_1, \dots, m_k}$ where $1\le m_1\le \dots \le m_k$ denotes the sizes of parts, $m_1+ \dots + m_k =m$, and $k\ge2$ is the number of parts (see \cite[Theorem 6.7]{CvDS}). 

To prove a), let us consider a graph $G_A$ different from any complete multipartite graph $K_{m_1, \dots, m_k}$. Therefore, $\lambda_2(A)>0$. We combine this information with the result due to D. Powers regarding the second largest eigenvalue $\lambda_2(A)$. According to \cite{Pow} (see also \cite{Powers1987}, \cite{Powers1989}), for a simple connected graph $G_A$ on $m$ vertices we have the following estimate for the second largest eigenvalue $\lambda_2(A)$:
\[
-1\le \lambda_2(A)\le \lfloor m/2\rfloor -1
\]
(see also Cvetkovi\'{c} and Simi\'{c} \cite{Cvetkovic1995}). Since $\lambda_2(A)>0$ we have $0<\lambda_+(A) \le  \lambda_2(A)\le \lfloor m/2\rfloor -1$, and $- \sqrt{\lfloor m/2\rfloor \lceil m/2\rceil} \le \lambda_{min}(A)\le \lambda_-(A)<0$. Hence the spectral gap $\Lambda^{gap}=\lambda_+(A)-\lambda_-(A)\le
\sqrt{\lfloor m/2\rfloor \lceil m/2\rceil} + \lfloor m/2\rfloor -1$.  If $m$ is even, it leads to the estimate $\Lambda^{gap}\le m-1$. If $m$ is odd, then it is easy to verify $\Lambda^{gap}\le m-3/2$. Analogously, $\Lambda^{ind}\le m/2$ if $m$ is even, and $\Lambda^{ind}\le \sqrt{m^2-1}/2$ if $m$ is odd.

The part b) is contained in Section 3 dealing with statistical properties of eigenvalue indices. 
\end{proof}

\medskip

Recall that for the complete bipartite graph $K_{m,m}$ the spectrum consists of zeros and $\pm m$. As a consequence $\lim_{m\to \infty} \Lambda^{gap}(K_{m,m}) = \infty$. The next result shows that a small change in a large graph $K_{m,m}$ caused by the removal of a single edge may result in a huge change in the spectral gap. 

\begin{proposition}
\label{Km1m2-e}
Let us denote by $K_{m,m}^{-e}$ the bipartite noncomplete graph constructed from the complete bipartite graph $K_{m,m}$ by deleting exactly one edge. Then its spectrum consists of $2m -4$ zeros and four real eigenvalues 
\begin{equation}
\lambda^{\pm,\pm}=\pm\left( 1-m \pm \sqrt{m^2 + 2m -3}\right)/2.
\label{lambdapmpm}
\end{equation}
For the spectral gap we have $\Lambda^{gap}(K_{m,m}^{-e})= 1-m + \sqrt{m^2 + 2m -3}$, and
\[ 
2\sqrt{1-2/(m+1)} < \Lambda^{gap}(K_{m,m}^{-e}) < 2\sqrt{1-1/m}. 
\]
As a consequence, $\lim_{m\to \infty} \Lambda^{gap}(K_{m,m}^{-e}) =2$.
\end{proposition}

\begin{proof}
Without loss of generality, we may assume that the adjacency matrix $A$  of the graph $K_{m,m}^{-e}$ has the form
\[
A = \left(
\begin{array}{cc}
0 & \mathbf{1}\mathbf{1}^T\\
\mathbf{1}\mathbf{1}^t & 0
\end{array}
\right) 
- 
\left(
\begin{array}{c}
0 \\
e_1
\end{array}
\right) (e_1, 0) 
- 
\left(
\begin{array}{c}
e_1 \\
0
\end{array}
\right) (0, e_1),
\]
where $\mathbf{1}=(1,\dots,1)^T, e_1=(1,0,\dots, 0)^T\in\mathbb{R}^m$. Assume that $\lambda$ is an eigenvalue of $A$, and $(0,0)\not=(x,y)\in \mathbb{R}^m\times\mathbb{R}^m$ is an eigenvector. Denote $\alpha=\sum_{i=1}^m x_i, \ \beta=\sum_{i=1}^m y_i$. Then
\[
\beta - y_1 = \lambda x_1,\quad  \alpha - x_1 = \lambda y_1, \quad \beta=\lambda x_i, \quad \alpha=\lambda y_i, \quad  i=2,\dots,m. 
\]
Assuming $\lambda=\pm 1$ leads to an obvious contradiction, as it implies $\alpha=\beta=0$, and $x=0, y=0$. The matrix $A$ has zero eigenvalue $\lambda=0$, with $2(m-1)$ dimensional eigenspace $\{ (x,y)\in \mathbb{R}^m\times \mathbb{R}^m, x_1=y_1=0\}$. Therefore, for $\lambda\not=\pm 1, 0$ we have  $x_1=(\alpha - \beta \lambda)/(1-\lambda^2), \ y_1=(\beta - \alpha \lambda)/(1-\lambda^2)$, and $x_2=\beta/\lambda, y_i=\alpha/\lambda, , \quad  i=2,\dots,m$. It results in a system of two linear equations for $\alpha,\beta$:
\[
\alpha =\frac{m-1}{\lambda}\beta + \frac{\alpha - \beta \lambda}{1-\lambda^2}, 
\quad
\beta =\frac{m-1}{\lambda}\alpha + \frac{\beta - \alpha \lambda}{1-\lambda^2}, 
\]
which has a non-trivial solution $(\alpha,\beta)\not=(0,0)$ provided that $\lambda\not=\pm1,0,$ is a solution of the following dispersion equation:
\[
\left( \frac{1}{1-\lambda^2} - 1 \right)^2 - \left( \frac{m-1}{\lambda} - \frac{\lambda}{1-\lambda^2} \right)^2 = 0.
\]
After rearranging terms, $\lambda$ is a solution of the cubic equation
\[
\pm \lambda^3 + m \lambda^2 -m +1 = 0,
\] 
having roots $\mp 1$ (which are not eigenvalues of $A$), and four other roots $\lambda^{\pm,\pm}$ given as in (\ref{lambdapmpm}), as claimed. The rest of the proof easily follows. 
\end{proof}

A similar property to the result of Proposition~\ref{Km1m2-e} regarding  indices can be observed when adding one edge to a complete bipartite graph, that is, destroying the bipartiteness of the original complete bipartite graph by small perturbation. 

\begin{proposition}
\label{Km1m2+e}
Let us denote by $G_A=K_{m,m}^{+e}$ a graph of the order $2m$ constructed from the complete bipartite graph $K_{m,m}$ by adding exactly one edge to the first part. Then its spectrum consists of $2m -4$ zeros and four real eigenvalues $\lambda^{(1),(2),(3),(4)}$ where $\lambda^{(4)}=\lambda_-(A)=-1$, and three other roots $\lambda^{(3)} <-1 < 0<\lambda^{(2)} <\lambda^{(1)}$ solve the cubic equation $\lambda^2(1-\lambda) - m(m-2 -m\lambda)=0$. 
The smallest positive eigenvalue has the form $\lambda_+(A)\equiv \lambda^{(2)} = 1-2/m - 2/m^3 +O(m^{-4})$ as $m\to \infty$. As a consequence, $\lim_{m\to \infty} \Lambda^{gap}(K_{m,m}^{+e}) =2$, and $\lim_{m\to \infty} \Lambda^{ind}(K_{m,m}^{+e}) =1$.
\end{proposition}

\begin{proof}
It is similar to the proof of the previous Proposition~\ref{Km1m2-e}. Arguing similarly as before, one can show that $\lambda^{(4)}=-1$ is an eigenvalue with multiplicity one. The other nonzero eigenvalues are roots of the cubic equation $\lambda^2(1-\lambda) - m(m-2 -m\lambda)=0$ which can be transformed into a depressed cubic equation with a positive discriminant $\Delta$. Thus, it has three distinct real eigenvalues $\lambda^{(1),(2),(3)}$. Performing the standard asymptotic analysis, we conclude $\lambda_+(A)=\lambda^{(2)} = 1-2/m - 2/m^3 +O(m^{-4})$ as $m\to\infty$, as claimed.
\end{proof}

\begin{remark}
In \cite{FoPi} it is shown that for a bipartite graph $K_{m_1, m_2}$ of the order $m=m_1+m_2$ and the average valency $d$ of vertices, one has $\lambda_{m/2} - \lambda_{1+m/2}\le \sqrt{d}$. 
\end{remark}

We end this section with the following statement regarding the density of values of the spectral index $ \Lambda^{gap}$ in the class of complete bipartite graphs. 

\begin{proposition}
For every pair of real numbers $0\le \delta<\gamma<1$, there exist an order $m$ and a complete bipartite graph $K_{m_1, m_2}$ of the order $m=m_1+m_2$ such that \ $m-\gamma \le \Lambda^{gap}(K_{m_1, m_2}) \le m-\delta$.
\end{proposition}

\begin{proof}
Recall the known fact (see, e.g. \cite{ElMc}) that the set of fractional parts $\sqrt{m} - [\sqrt{m}]$ of roots of all positive integers $m$ is dense in the interval $[0,1)$. Hence, there exists an integer $m_2$, such that $\sqrt{\delta}\le \sqrt{m_2}- [\sqrt{m_2}]\le \sqrt{\gamma}$. Take $m_1:=[\sqrt{m_2}]^2 \le m_2$. Then $\sqrt{\delta}\le \sqrt{m_2}- \sqrt{m_1}\le \sqrt{\gamma}$. By squaring and rearranging terms, we obtain $(m_1+m_2)-c\le 2\sqrt{m_1 m_2} \le (m_1+m_2) -d$. Now we take the bipartite graph $K_{m_1,m_2}$, of order $m=m_1+m_2$. Since $\Lambda^{gap}(K_{m_1, m_2}) = 2\sqrt{m_1 m_2}$ the claim follows.
\end{proof}

\subsection{ Indices for noncomplete graphs}

The purpose of this section is to analyze  indices for noncomplete multipartite graphs. 

\begin{proposition}
If $G_A$ is a bipartite but not complete bipartite graph, with the average vertex degree $d$, and the multiplicity of the zero eigenvalue of the order $k$, then 
\begin{equation}
\label{eq:bip} 
\Lambda^{gap}(G_A) \le 2\sqrt{\frac{d(m-2d)}{m-k-2}}\ .
\end{equation}  

\end{proposition}

\begin{proof}
Let $G_A$ be a bipartite but not complete bipartite graph with adjacency matrix $A$ having null space of dimension $k$. Since $G_A$ is not complete bipartite, we have $k\le m-4$. It follows that $m$ and $k$ have the same parity, so that $m-k=2r$ for some positive integer $r\ge 2$. By bipartiteness of $G_A$ we may assume that its eigenvalues have the form $\lambda_1\ge \lambda _2 \ge \ldots \ge \lambda_r > 0=\lambda_{r+1}= \ldots = \lambda_{r+k} > -\lambda_r \ge \ldots\ge -\lambda_2\ge -\lambda_1$, so that $\lambda_{+}=\lambda_r$ and $\lambda_{-} = -\lambda_r$. The earlier used fact that $\lambda_1\ge d$ trivially implies that 
\begin{equation}\label{eq:rs1a}
\sum_{i=1}^r\lambda_i^2 \ge d^2 + (r-1)\lambda_{+}^2\ \ .
\end{equation} 

It is well known that the sum of squares $\sum_{i=1}^m\lambda_i^2 = trace(A^2)= m d$, where $d$ is the average valency of vertices of $G_A$, that is, $m d/2$ is the number of edges in the graph $G_A$ (cf. Bapat \cite{Bapat2010}). Combined with the inequality $\lambda_1\ge d$ used earlier, we obtain
\begin{equation}\label{eq:rs1}
m d = 2\sum_{i=1}^r\lambda_i^2 \ge 2d^2 + 2(r-1)\lambda_{+}^2 = 2d^2 + (m-k-2)\lambda_{+}^2  
\end{equation}
and evaluation of $\lambda_{+}(A)$ from (\ref{eq:rs1}) gives $\lambda_{+}(A) = -\lambda_{-}(A) \le \sqrt{\frac{d(m-2d)}{m-k-2}}$ which implies the inequality (\ref{eq:bip}) in our statement.
\end{proof}

\begin{remark}
The estimate (\ref{eq:bip}) is nearly optimal. For example, for the graph $K_{m_1,m_1}^{-e}$ we have $m=2m_1$, $d= m_1- \frac{1}{m_1}$ and $k=m-4$, and (\ref{eq:bip}) for these values gives $\Lambda^{gap}(K_{m_1, m_1}^{-e})\le 2\sqrt{1-4/m^2}$, which is a slightly worse estimate than the one derived in the analysis of the spectrum of $K_{m_1,m_1}^{-e}$. 
\end{remark}

Finally, we show that the maximal (minimal) eigenvalue can increase (decrease) by adding one vertex to the original graph. 

\begin{proposition}\label{increaselambdamax}
Assume $G_A$ is a simple connected graph on the vertices $m$ with the maximal and minimal eigenvalues $\lambda_{max}(A)$, and $\lambda_{min}(A)$. Then there exists a graph $G_{\mathscr{A}}$ on the $m+1$ vertices constructed from $G_A$ by adding one vertex connected to each of the vertices $G_A$ that has the maximal eigenvalue such that 
\[
\lambda_{max}(\mathscr{A}) \ge  \frac{\lambda_{max} (A)+\sqrt{ (\lambda_{max}(A))^2 + 4}}{2}.
\]
Similarly, there exists a vertex $i_0$ of $G_A$ such that the graph $G_{\mathscr{A}}$ on $m+1$ vertices constructed from $G_A$ by adding a pendant vertex to the vertex $i_0$ has the minimal eigenvalues satisfying the estimate
\[
\lambda_{min}(\mathscr{A}) \le  \frac{\lambda_{min}(A) - \sqrt{ (\lambda_{min}(A))^2 + 4/m }}{2} .
\]
\end{proposition}

\begin{proof}
The sum of all eigenvalues of the symmetric matrix $A$ is zero because the trace of $A$ is zero. Hence $\lambda_{min}(A)<0<\lambda_{max}(A)$. Let $\mathscr A$ be the $(m+1)\times (m+1)$ adjacency matrix of the graph $G_{\mathscr{A}}$ obtained from $G_A$ by adding a vertex connected to a subset of vertices of $G_A$. Its adjacency matrix $\mathscr{A}$
has the block form
\begin{equation}
\mathscr{A} = \left(
\begin{array}{cc}
A & e\\
e^T & 0
\end{array}
\right),
\label{blockmatrixGk}
\end{equation}
where $e=(e_1, \dots, e_m)^T$, $e_{i}\in \{0,1\}$. The maximal eigenvalue $\lambda_{max}(\mathscr{A})$ can be computed by means of the Rayleigh ratio, i.e. 
\[
\lambda_{max}(\mathscr{A}) 
= 
\max_{x\in\R^m, \xi\in \R} 
\frac{
(x^T,\xi) \left(
\begin{array}{cc}
A & e\\
e^T & 0
\end{array}
\right)  \left(
\begin{array}{c}
x\\
\xi
\end{array}
\right)}
{|x|^2 +\xi^2}
=
\max_{x\in\R^m, \xi\in \R} 
\frac{ x^T A x + 2(e^T x) \xi}
{|x|^2 +\xi^2} ,
\]
where $|x|$ is the Euclidean norm of the vector $x$. Let $\hat x$ be an eigenvector for corresponding to the maximal eigenvalue $\lambda_{max}(A)$, that is, $A \hat x =  \lambda_{max} (A)\hat x$. Then 
\[
\lambda_{max}(\mathscr{A}) 
\ge 
\max_{\xi\in \R} 
\frac{\lambda_{max} (A)+ 2 (e^T \hat{x}) \xi}{1 +\xi^2}
=
\lambda_{max} (A)\max_{\xi\in \R} 
\frac{1 + \alpha \xi}{1 +\xi^2},
\]
where $\alpha = 2(e^T \hat{x})/\lambda_{max}(A)$.  Let us introduce the auxiliary function $\psi:\R\to\R$, $\psi(\xi)= (1+\alpha\xi)/(1+\xi^2)$, where $\alpha\in\R$ is a parameter. Using the first-order necessary condition it is easy to verify that the maximum of the function $\psi$ is attained at $\xi= (-1+\sqrt{1+\alpha^2})/\alpha$. As a consequence, we have
\[
\max_\xi \frac{1+\alpha\xi}{1+\xi^2} = \frac{1+\sqrt{1+\alpha^2}}{2} >0.
\]
Notice that the adjacency matrix contains only nonnegative elements. With regard to the Perron-Frobenius theorem, an eigenvector corresponding to the maximal eigenvalue $\lambda_{max}(A)$ is nonnegative, i.e. $\hat x\ge 0$. Consider the vector $e=(1, \dots, 1)^T$ consisting of ones. It corresponds to the new vertex connected to all the vertices of $G_A$. Then $(e^T \hat{x})^2 = (\hat x_1 + \dots + \hat x_m)^2 \ge |\hat x|^2=1$ because all $\hat x_i\ge 0$ are nonnegative.  Inserting the parameter $\alpha^2= 4(e^T \hat x)^2/(\lambda_{max}(A))^2 \ge 4 /(\lambda_{max}(A))^2$ we obtain $\lambda_{max}(\mathscr{A}) \ge  \frac12 (\lambda_{max} (A)+\sqrt{ (\lambda_{max}(A))^2 + 4 })$, as claimed.

Similarly, let $\bar{x}$ be the unit eigenvector corresponding to the minimal eigenvalue $\lambda_{min}(A)$, that is, $A \bar{x} = \lambda_{min}(A) \bar{x}, |\bar{x}|=1$.
Let $i_0$ be the index such that $|\hat x_{i_0}| = \max_i |\hat x_i |$. Since $|\hat x| =1$ we have $|\hat x_{i_0}|\ge 1/\sqrt{m}$. Assume that the graph $G_{\mathscr{A}}$ is constructed from $G_A$ by adding one vertex connected to the vertex $i_0$. That is $e=(e_1, \dots, e_m)^T$, $e_{i_0}=1$, and $e_i=0$ for $i\not=i_0$. Then $(e^T \hat x)^2 = (\hat x_{i_0})^2\ge 1/m$. Hence 
\[
\lambda_{min}(\mathscr{A}) 
= 
\min_{x\in\R^m, \xi\in \R} 
\frac{ x^T A x + 2(e^T x) \xi}
{|x|^2 +\xi^2} 
\le 
\min_{\xi\in \R} 
\frac{ \lambda_{min}(A) + 2(e^T \bar{x}) \xi}
{1 +\xi^2} 
=
\lambda_{min}(A) \max_{\xi\in \R} 
\frac{ 1 + \alpha \xi}
{1 +\xi^2}
\]
because $\lambda_{min}(A)<0$. Here $\alpha = 2(e^T \bar{x})/\lambda_{min}(A)$. Consider the index $i_0$ for which $|x_{i_0}|$ is maximal. Then $(\bar{x}_0)^2 \ge 1/m$, and 
\[
\lambda_{min}(\mathscr{A}) \le \lambda_{min}(A) \frac{1+\sqrt{1+\alpha^2}}{2}
\le \frac{\lambda_{min}(A) - \sqrt{(\lambda_{min}(A))^2 + 4/m}}{2},
\]
and the proof of the proposition follows.
\end{proof}

\section{Statistical properties of  indices}

The purpose of this section is to report statistical results on maximal (minimal) eigenvalues, and  indices for the class of all simple connected graphs on $m\le 10$ vertices. 
In Table~\ref{tab-spektrumall} the operators $E,\sigma, {\mathcal S}$ and ${\mathcal K}$ represent the mean value, standard deviation, skewness and kurtosis of the corresponding sets of eigenvalues $\lambda_{max}$, and $\lambda_{min}$, respectively. For larger $m$ the skewness ${\mathcal S}(\lambda_{max})$ approaches zero and the kurtosis ${\mathcal K}(\lambda_{max})$ tends to $3$ meaning that the distribution of maximal eigenvalues of all simple connected graphs on the $m$ vertices becomes normally distributed as $m$ increases. The skewness ${\mathcal S}(\lambda_{min})<0$ is negative and the kurtosis ${\mathcal K}(\lambda_{min})>3$ meaning that the distribution of minimal eigenvalues of connected graphs on the $m$ vertices is skewed to the left. It has fat tails (leptokurtic distribution) because it has positive excess kurtosis ${\mathcal K}(\lambda_{min})-3 >0$ as $m$ increases.
We employed the list of all simple connected graphs due to B. McKay which is available at the repository \cite{McKay}. We calculated the spectra for all graphs and the corresponding  indices. Calculating  indices for $m=10$ is a computationally complex task, since the number $11716571$ of all simple connected graphs is very large. To our knowledge, a consolidated list of connected nonisomorphic graphs is not available for orders $m\ge 11$. 

\begin{table}
\caption{\small Descriptive statistics of the maximal(minimal) eigenvalues $\lambda_{max}$ ($\lambda_{min}$), spectral gap $\Lambda^{gap}$, spectral index $\Lambda^{ind}$, and spectral power $\Lambda^{pow}$ for all simple connected graphs on $m\le 10$ vertices.}
\scriptsize
\begin{center}
\hglue -0.5truecm \begin{tabular}{l||l|l|l|l|l|l|l|l|l}
\hline
$m$  & $2$ & $3$ & $4$ & $5$ & $6$ & $7$ & $8$ & $9$ & $10$\\
total \# & $1$   & $2$   & $6$   &  $21$ & $112$ & $853$ &$11117$&$261080$& $11716571$\\
\hline\hline
$E(\lambda_{max})$             & 1 &  1.7071 & 2.1802 & 2.6417 & 3.0582 & 3.4856 & 3.9288 & 4.4001 & 4.8895\\
$\sigma(\lambda_{max})$        & 0 &  0.4142 & 0.5228 & 0.5968 & 0.6368 & 0.6562 & 0.6595 & 0.6529 & 0.6471\\
${\mathcal S}(\lambda_{max})$  & - &  0 & 0.5096 & 0.5171 & 0.4142 & 0.2855 & 0.1536 & 0.0608 & 0.0132\\
${\mathcal K}(\lambda_{max})$  & - &  1 & 1.9715 & 2.6351 & 2.9901 & 3.0804 & 3.0578 & 3.0313 & 3.0096\\
$\max(\lambda_{max})$          & 1 &  2        & 3        & 4        & 5        & 6        & 7        & 8        & 9\\
$\min(\lambda_{max})$          & 1 &  1.4142 & 1.6180 & 1.7321 & 1.8019 & 1.8478 & 1.8794 & 1.9021 & 1.9190\\
\hline
$E(\lambda_{min})$             & -1 & -1.2071 & -1.5655 & -1.7911 & -2.0302 & -2.2264 & -2.4191 & -2.6018 & -2.7756\\
$\sigma(\lambda_{min})$        &  0 &  0.2929 &  0.3305 &  0.2981 &  0.3012 &  0.2995 &  0.2994 &  0.2915 &  0.2832\\
${\mathcal S}(\lambda_{min})$  & -  &  0 &  0.5740 &  0.2506 & -0.4079 & -0.5438 & -0.4937 & -0.4121 & -0.3927\\
${\mathcal K}(\lambda_{min})$  & -  &  1 &  2.7899 &  4.2278 &  4.1917 &  3.5318 &  3.3933 &  3.3626 &  3.3289\\
$\max(\lambda_{min})$          & -1 & -1        & -1        & -1        & -1        & -1        & -1        & -1        & -1\\
$\min(\lambda_{min})$          & -1 & -1.4142 & -2        & -2.4495  & -3        & -3.4641 & -4        & -4.4721 & -5\\
\hline
$\max(\Lambda^{gap})$ & 2 & 3 & 4 & 5 & 6 & 7 & 8       & 9      & 10 \\
$\min(\Lambda^{gap})$ & 2 & 2.8284 & 1.2360 & 1.0806 & 0.7423 & 0.6390 & 0.3468 & 0.2834 & 0.1565 \\
$\max(\Lambda^{ind})$ & 1 & 2 & 3 & 4 & 5 & 6 & 7       & 8      & 9 \\
$\min(\Lambda^{ind})$ & 1 & 1.4142 & 0.6180 & 0.6180 & 0.4142 & 0.3573 & 0.1826 & 0.1502 & 0.0841 \\
$\max(\Lambda^{pow})$      & 2 & 4 & 6 & 8 &10 &12 & 14.3253 & 17.0600& 20 \\
$\min(\Lambda^{pow})$      & 2 &2.8284 & 3.4642 & 4.0000 & 4.4722 & 4.8990 & 5.2916 & 5.6568 & 6.0000 \\
\hline
\end{tabular}
\end{center}
\label{tab-spektrumall}
\end{table}

\begin{figure}
    \centering
    
    \includegraphics[width=.3\textwidth]{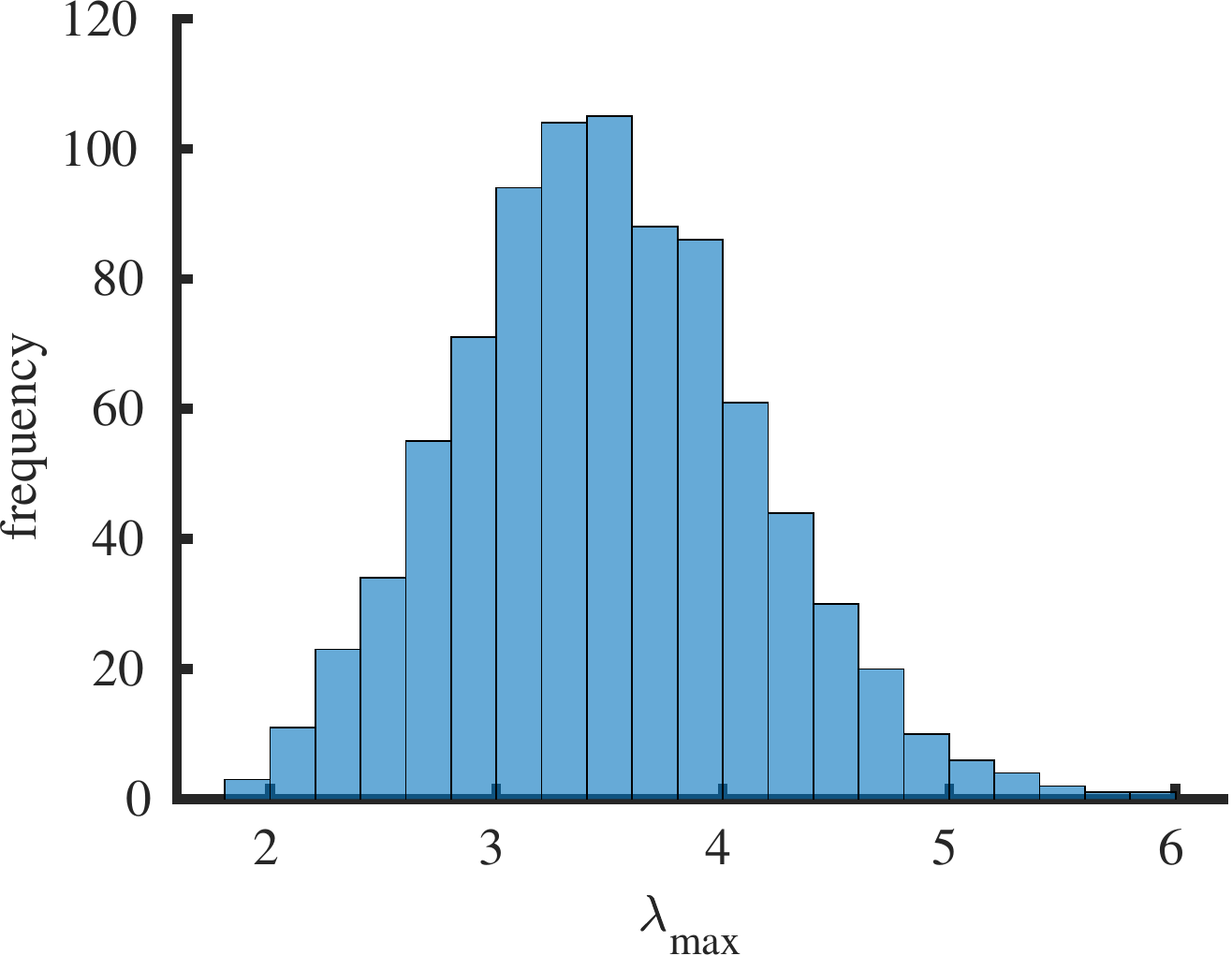}
    \quad
    \includegraphics[width=.3\textwidth]{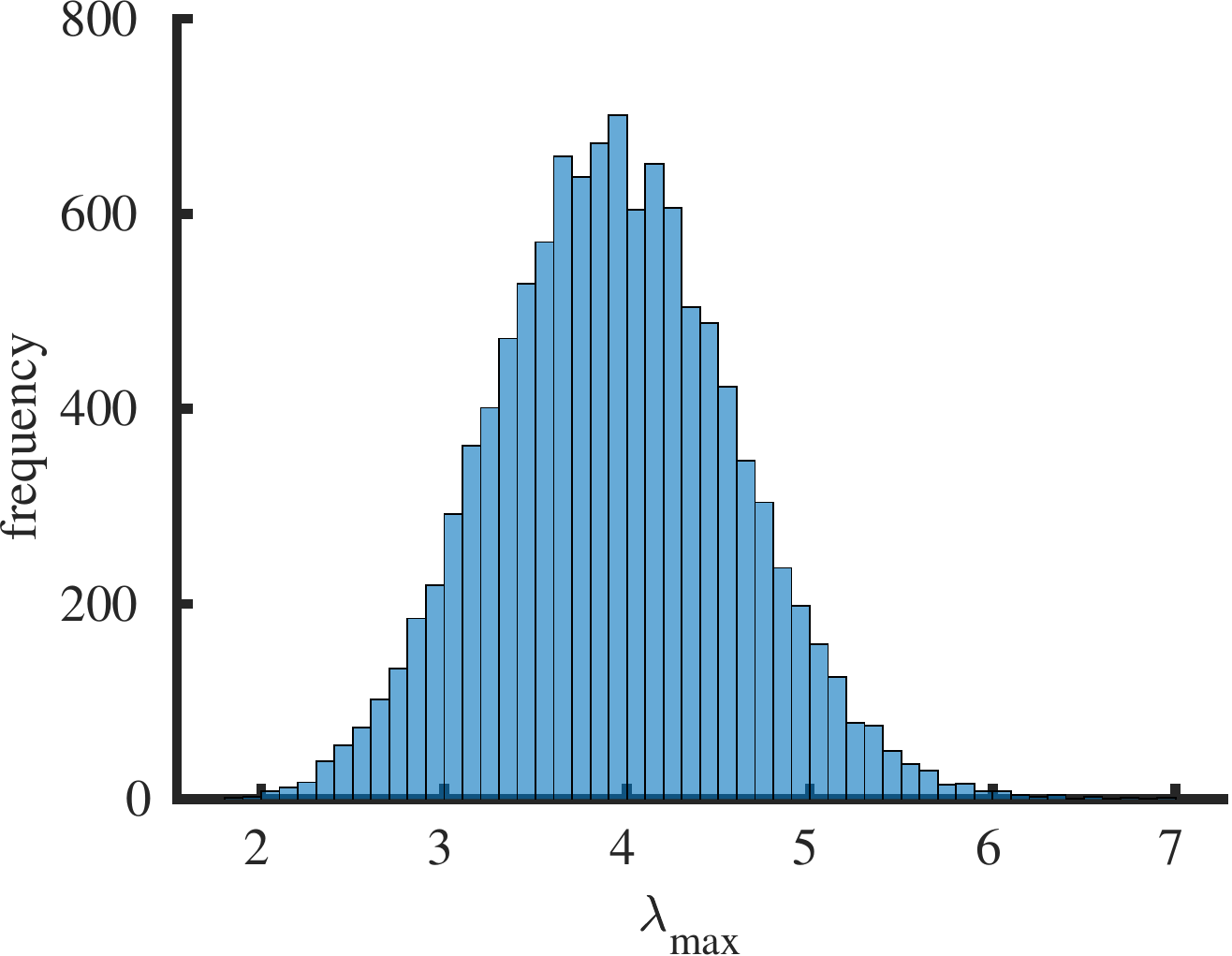}
    \quad
    \includegraphics[width=.3\textwidth]{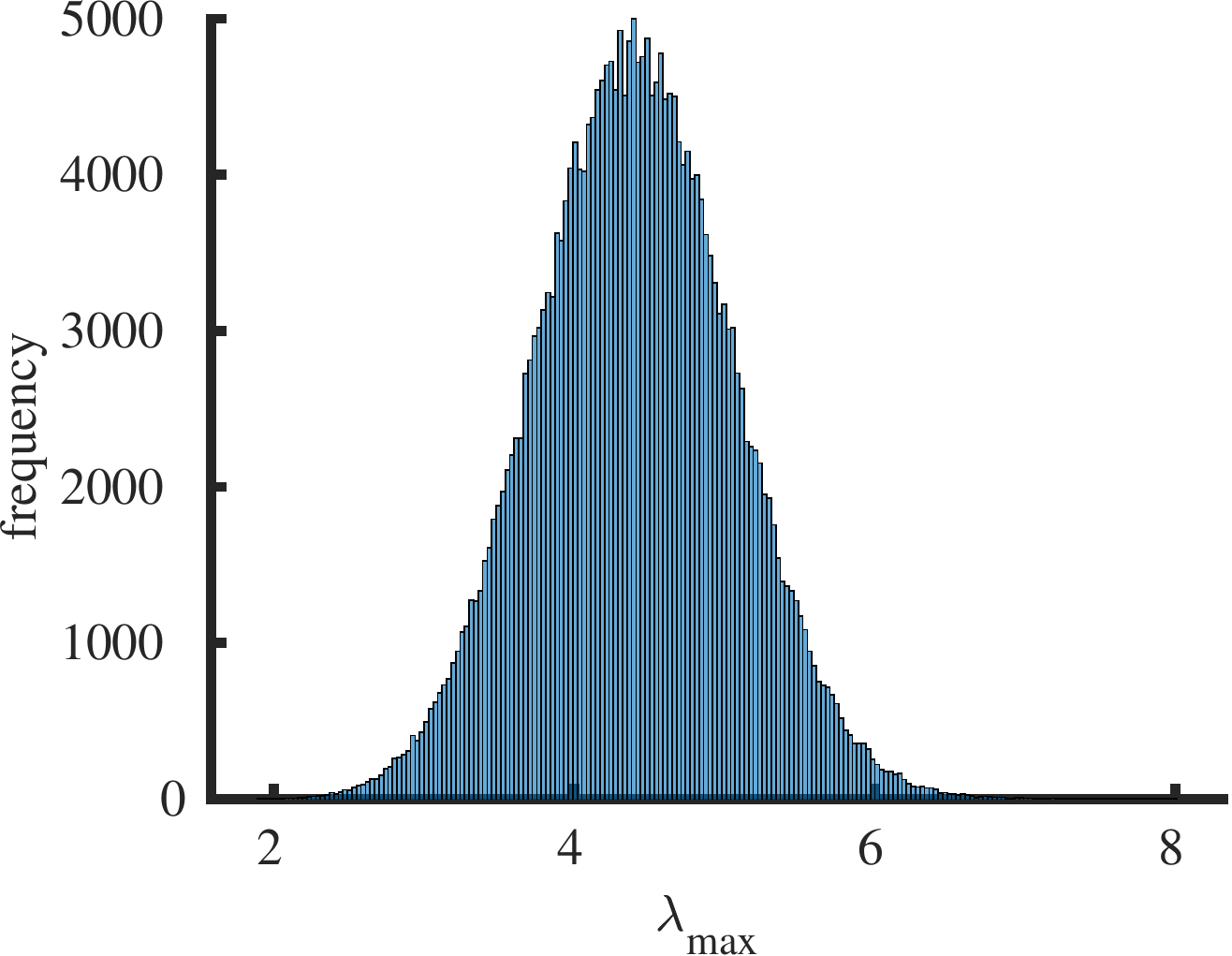}
    
{\small $m=7$ \hskip 4truecm $m=8$ \hskip 4truecm $m=9$}

\vskip 0.5truecm 

    \includegraphics[width=.3\textwidth]{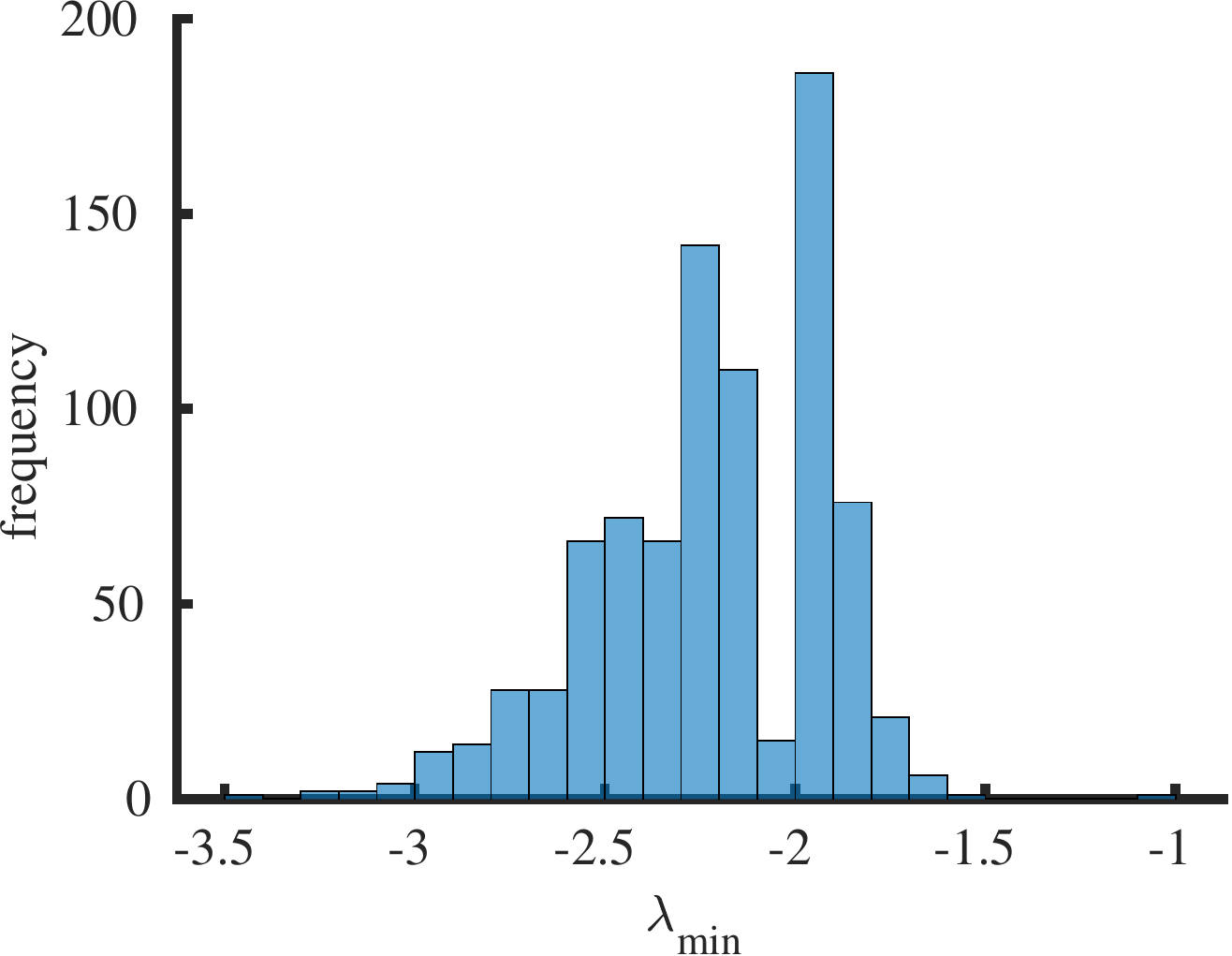}
    \quad
    \includegraphics[width=.3\textwidth]{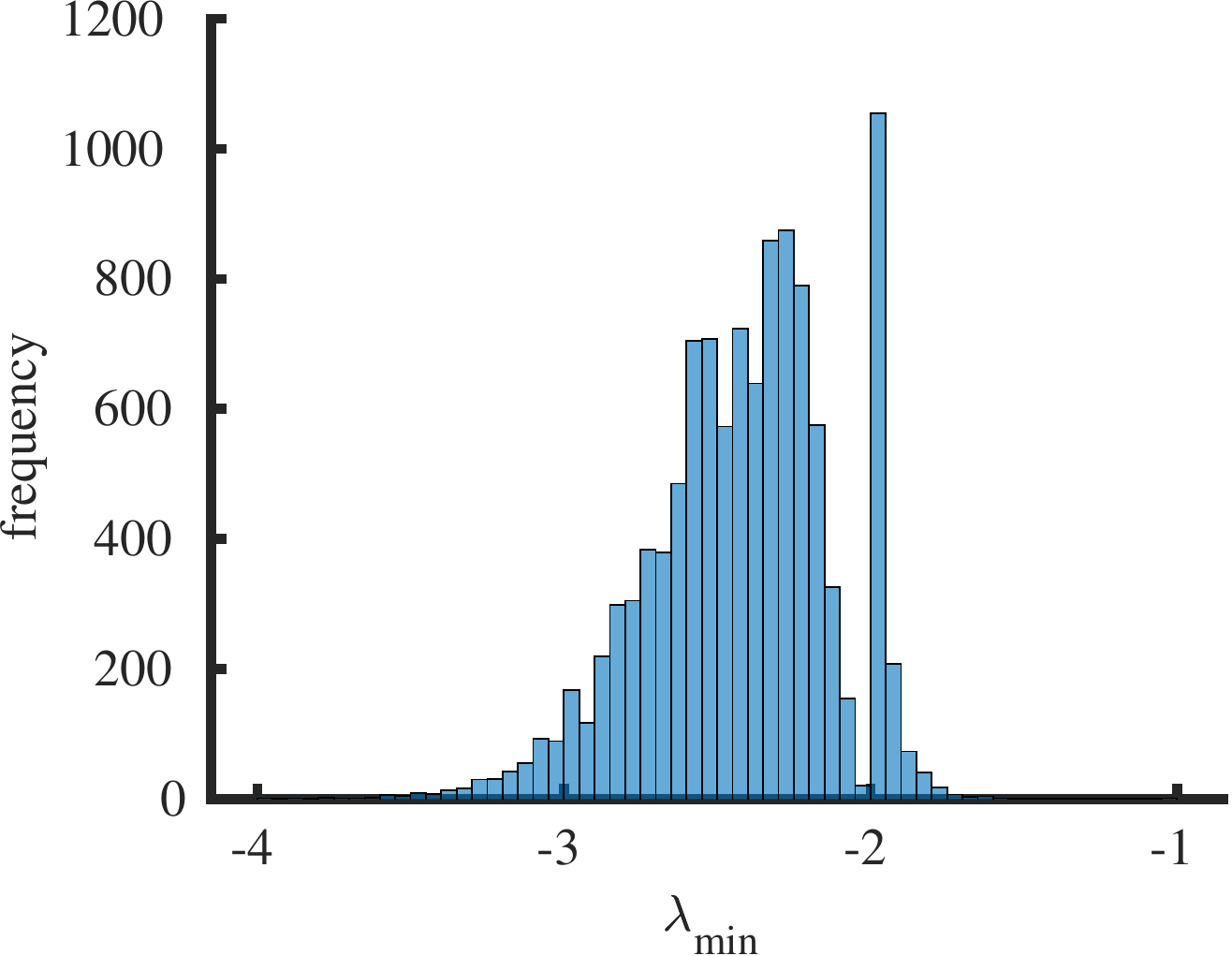}
    \quad
    \includegraphics[width=.3\textwidth]{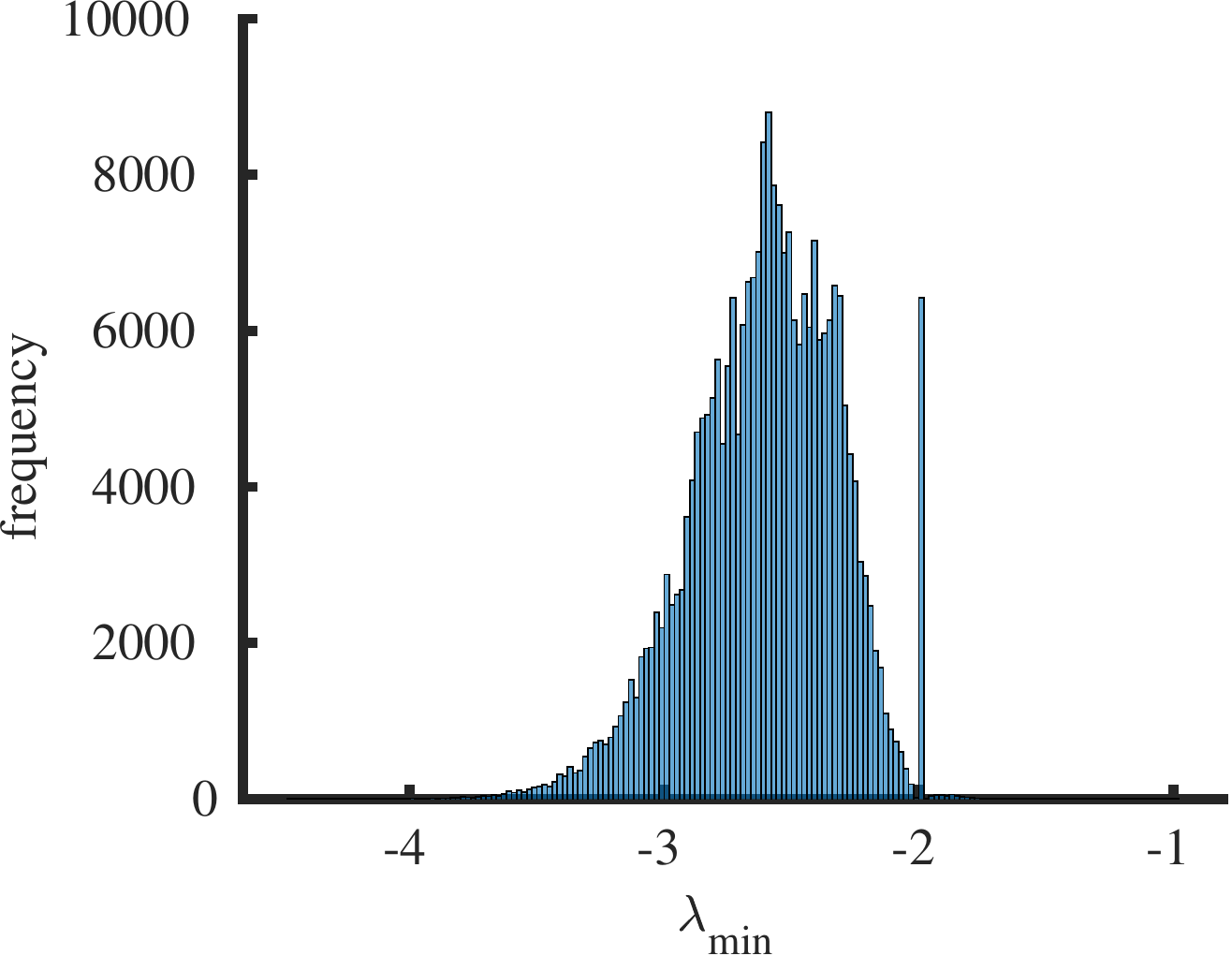}

{\small $m=7$ \hskip 4truecm $m=8$ \hskip 4truecm $m=9$}

    \caption{\small Histograms of distribution of maximal (top row) and minimal (bottom row) eigenvalues for all simple connected graphs on $7\le m\le 9$ vertices. For their statistical properties, see Table~\ref{tab-spektrumall}.}
    \label{fig-graf-maxlambdamaxmin-7-8-9}
\end{figure}

\begin{figure}
    \centering

    \includegraphics[width=.3\textwidth]{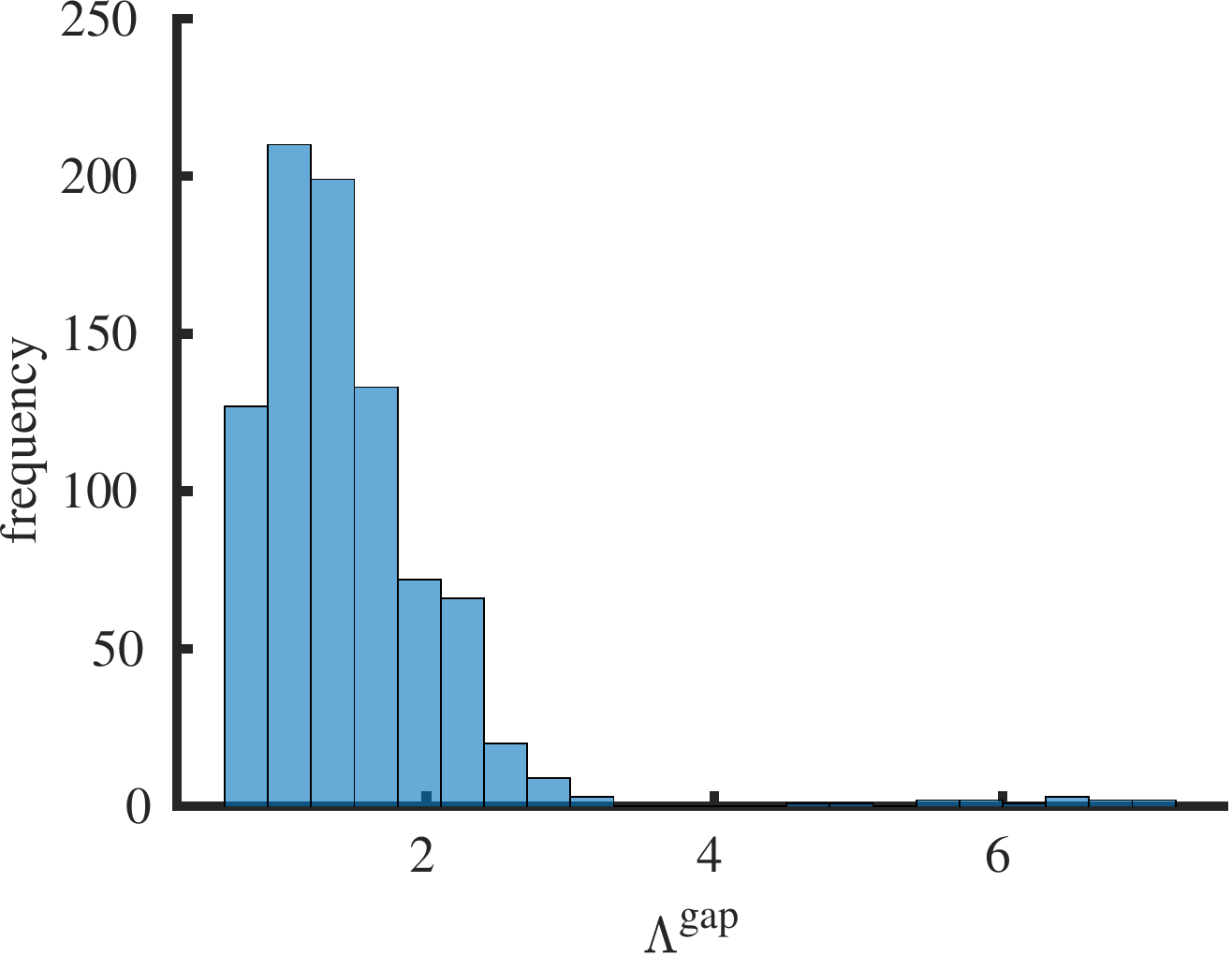}
    \quad
    \includegraphics[width=.3\textwidth]{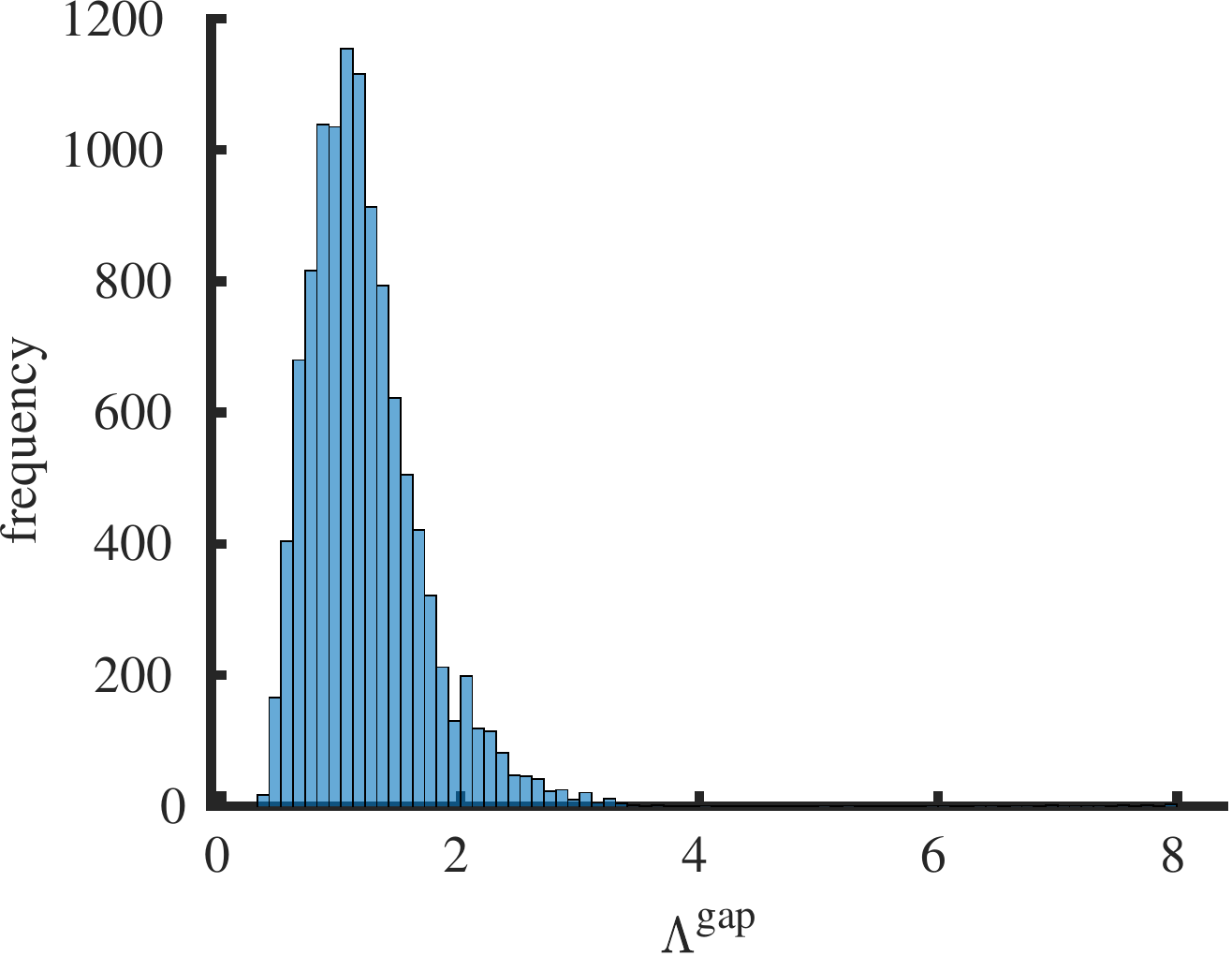}
    \quad
    \includegraphics[width=.3\textwidth]{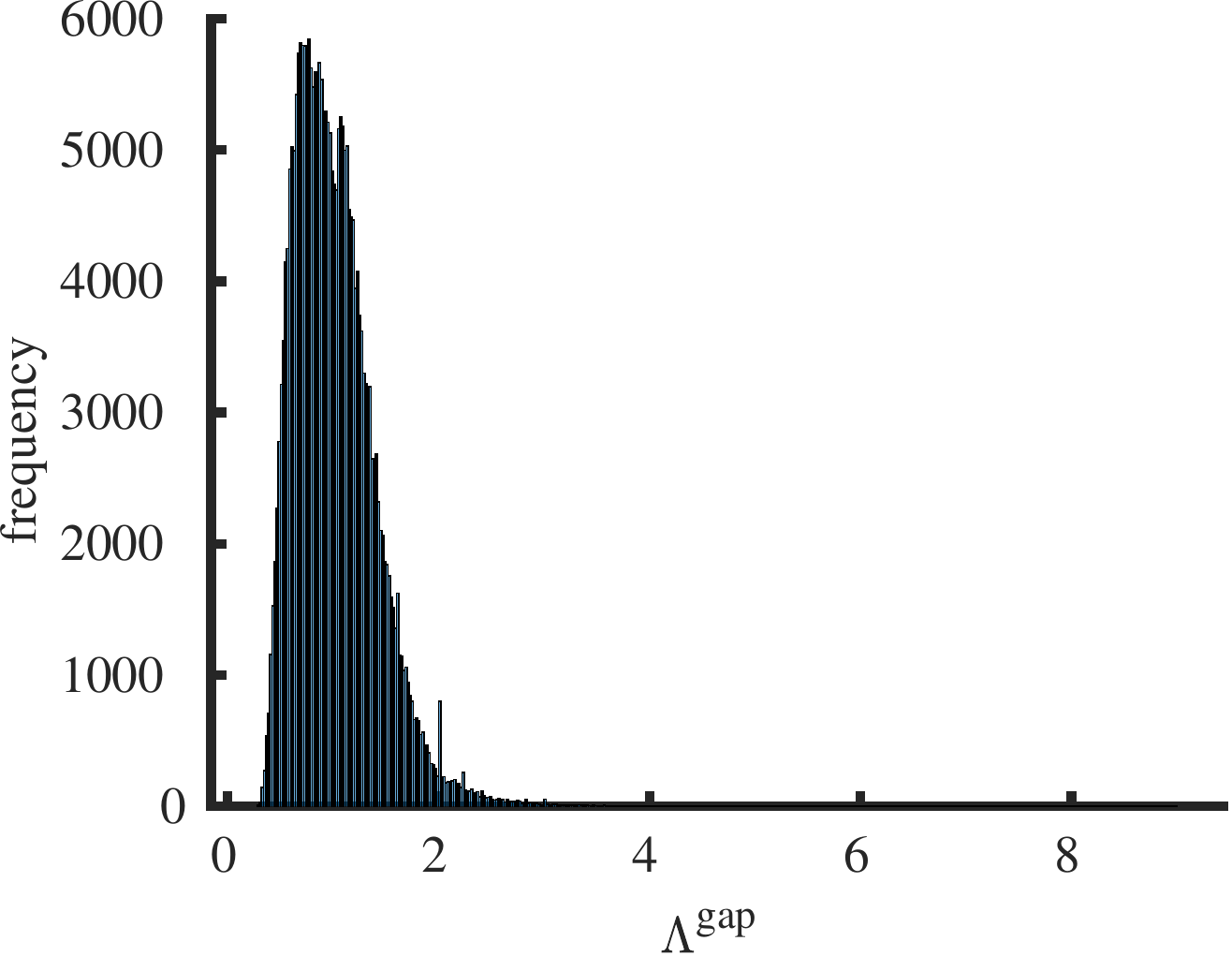}
    
    \includegraphics[width=.3\textwidth]{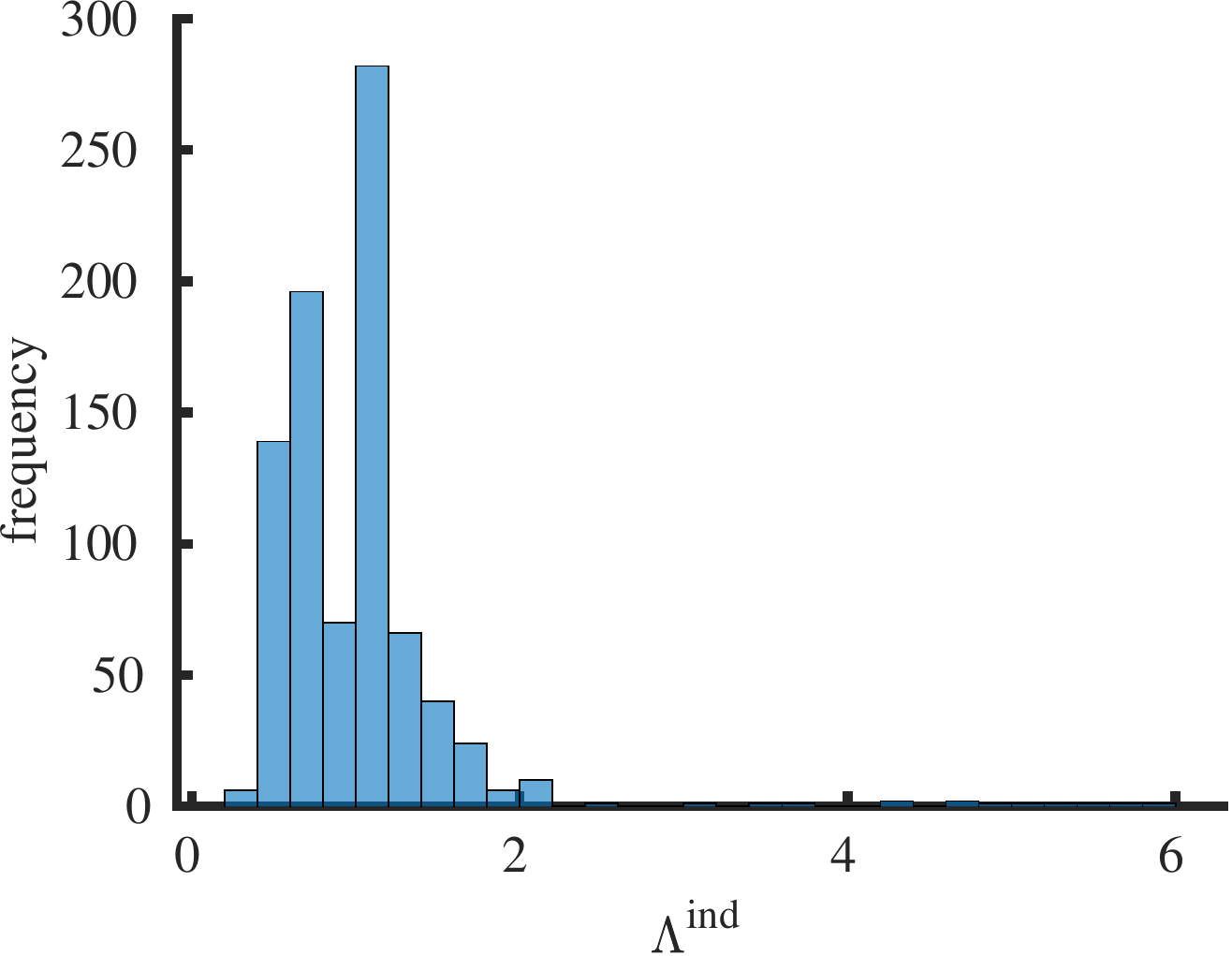}
    \quad
    \includegraphics[width=.3\textwidth]{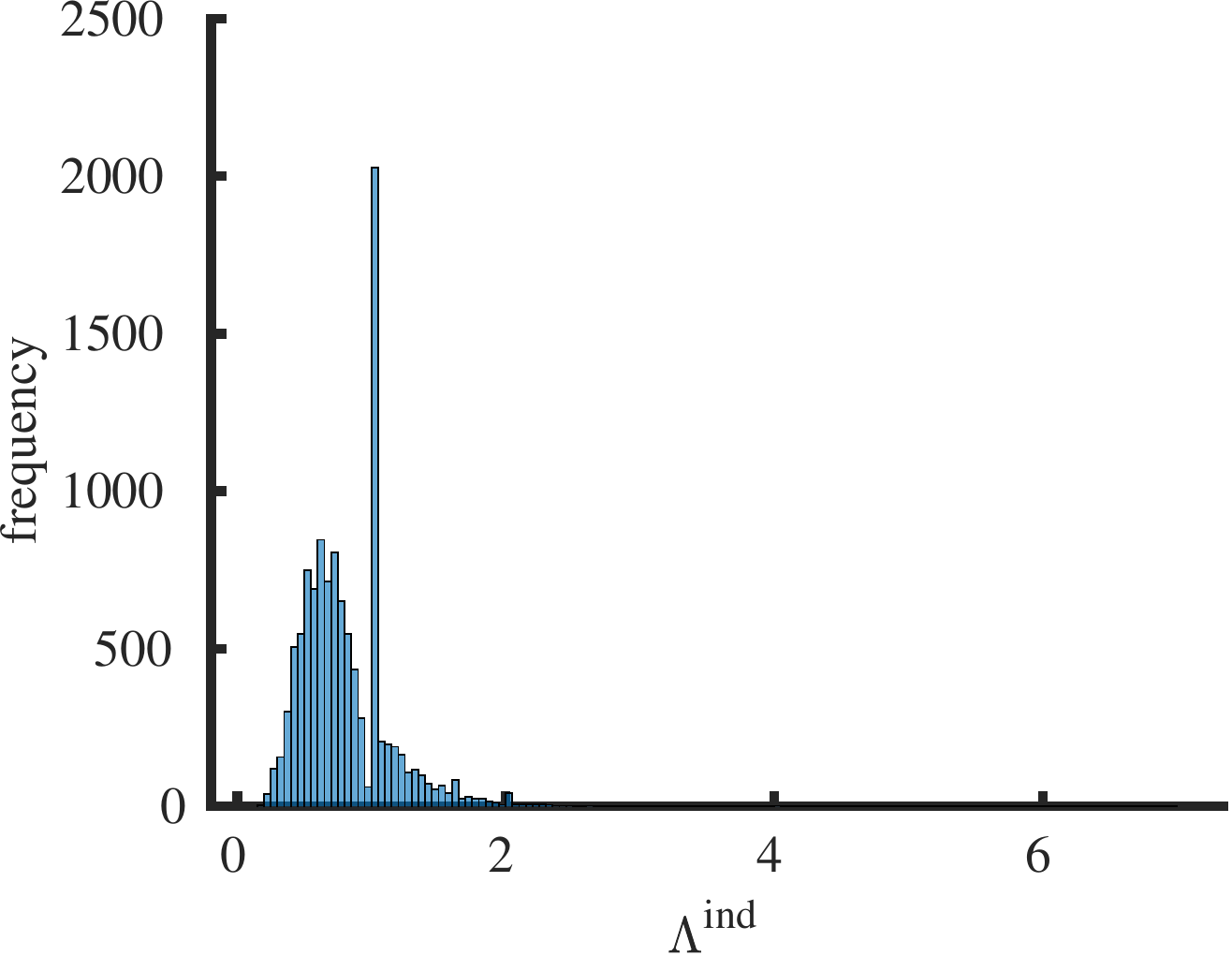}
    \quad
    \includegraphics[width=.3\textwidth]{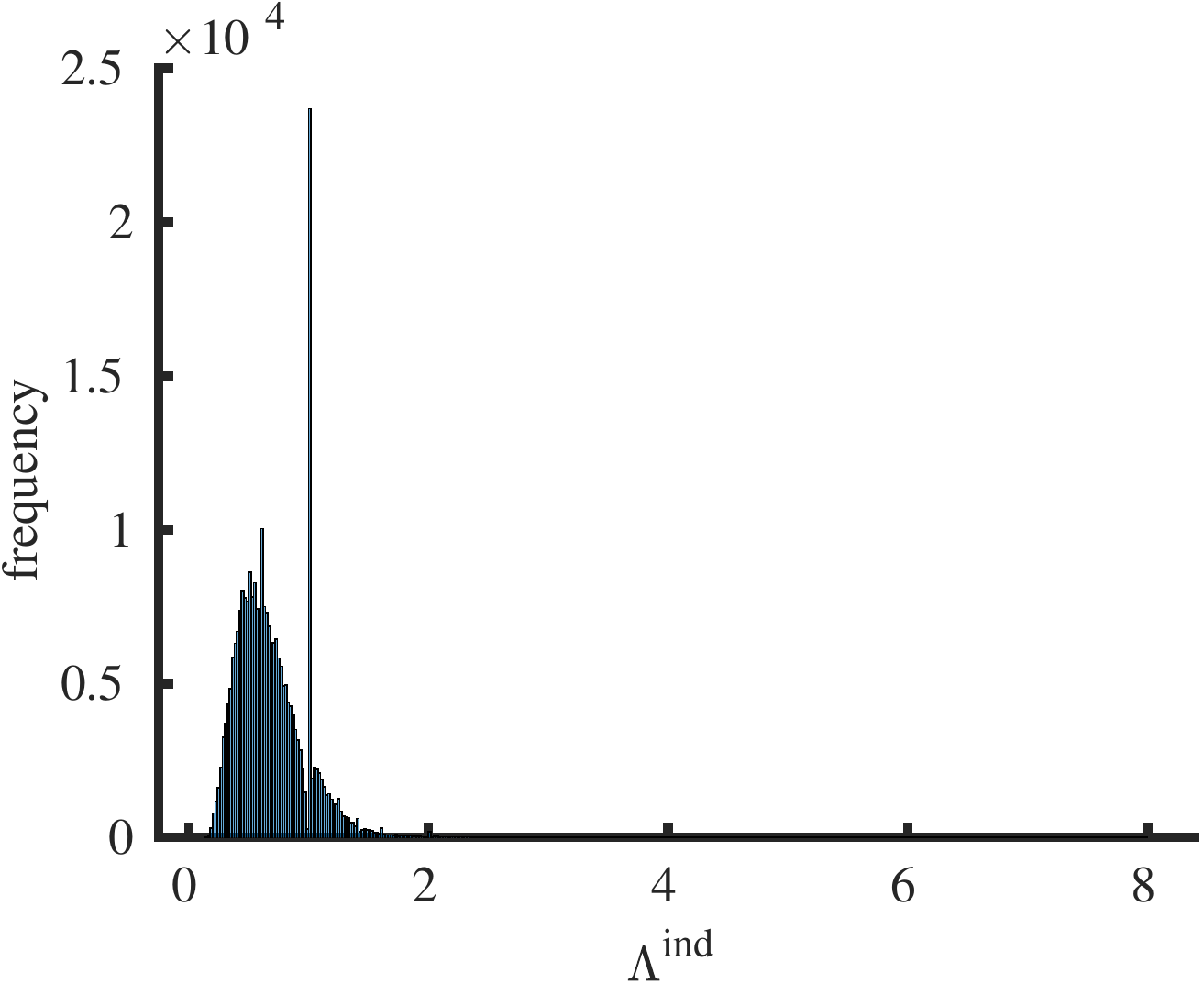}
    
    \includegraphics[width=.3\textwidth]{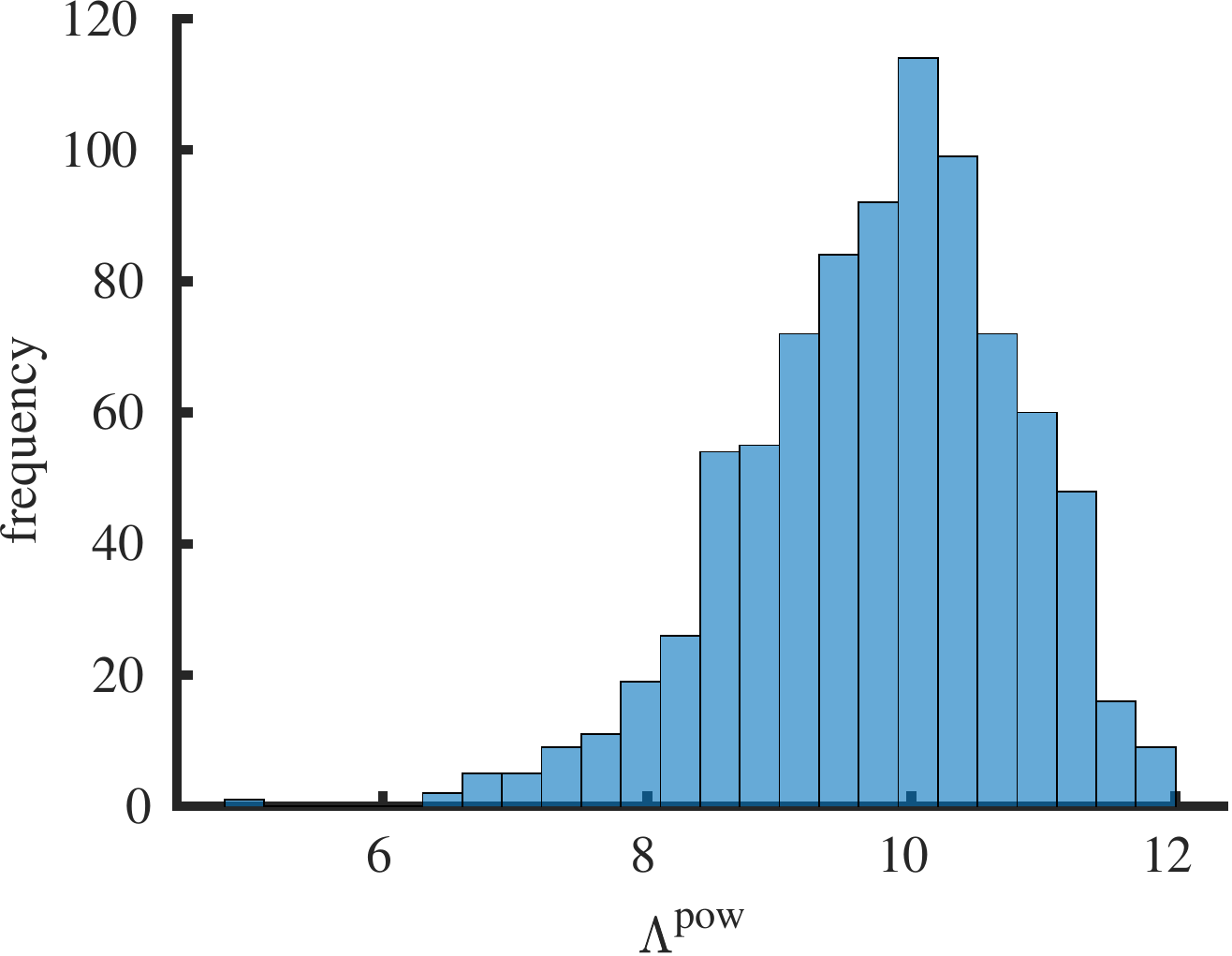}
    \quad
    \includegraphics[width=.3\textwidth]{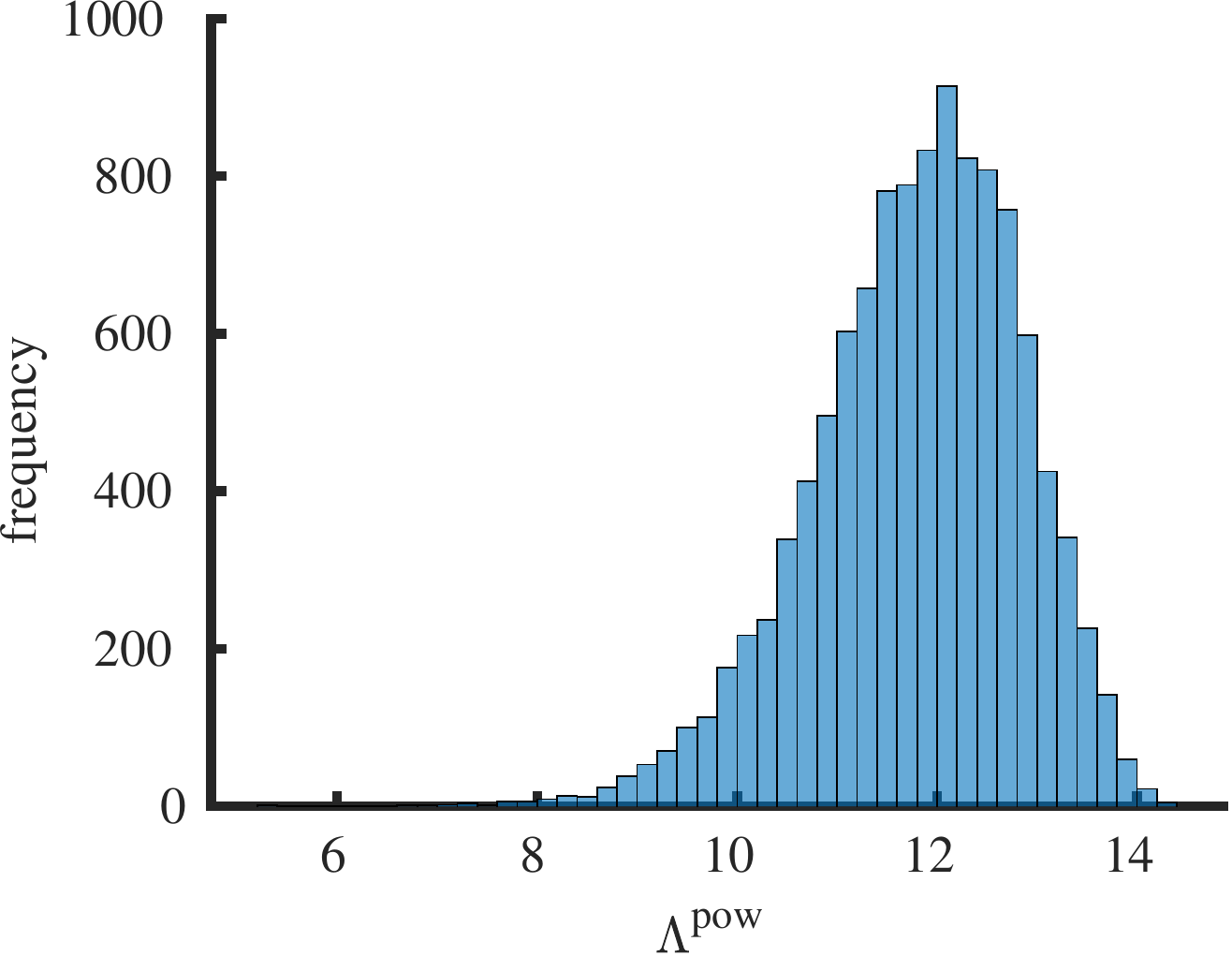}
    \quad
    \includegraphics[width=.3\textwidth]{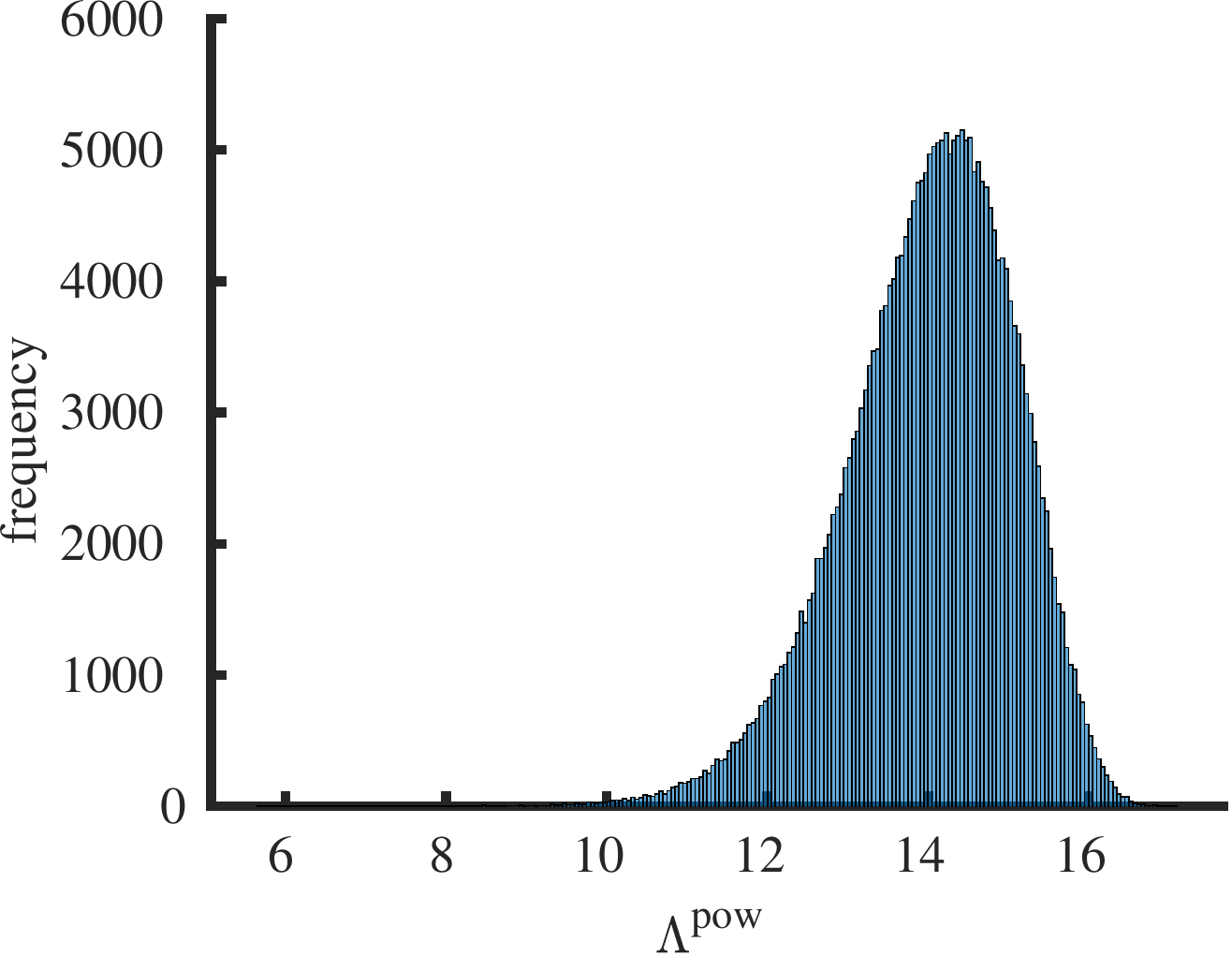}
    
{\small $m=7$ \hskip 3.5truecm $m=8$ \hskip 3.5truecm $m=9$}

    \caption{\small Histograms of distribution of $\Lambda^{gap}$ (top row), $\Lambda^{ind}$ (middle row), and $\Lambda^{pow}$ (bottom row) for all simple connected graphs on $7\le m\le 9$  vertices. For their statistical properties, see Table~\ref{tab-spektrumall}.}
    \label{fig-graf-maxlambaHL-power-7-8-9}
\end{figure}

\begin{figure}
    \centering
    
    \includegraphics[width=.25\textwidth]{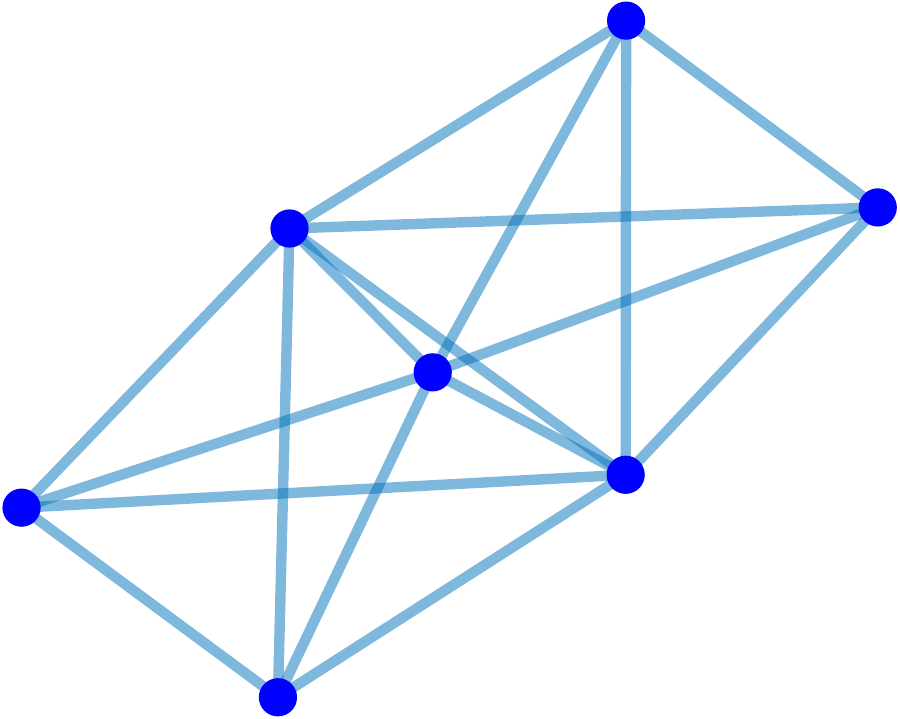}

    \caption{\small The noncomplete graph  on $m=7$  vertices with eigenvalues  $\{5, 1, -1, -1, -1, -1, -2\} $ maximizing the value $\Lambda^{pow}=12$ in the class of all simple connected graphs of the degree $m=7$.}
    
    \label{fig-graf-maxLambdaPower-2-second}
\end{figure}

\begin{figure}
    \centering
    
    \includegraphics[width=.25\textwidth]{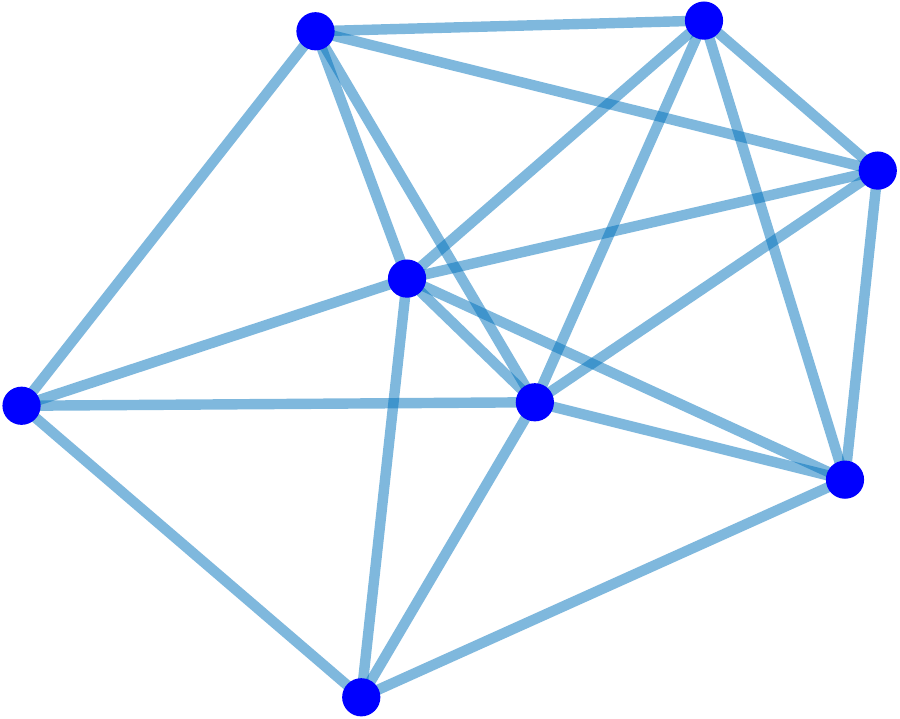}
    \quad
    \includegraphics[width=.25\textwidth]{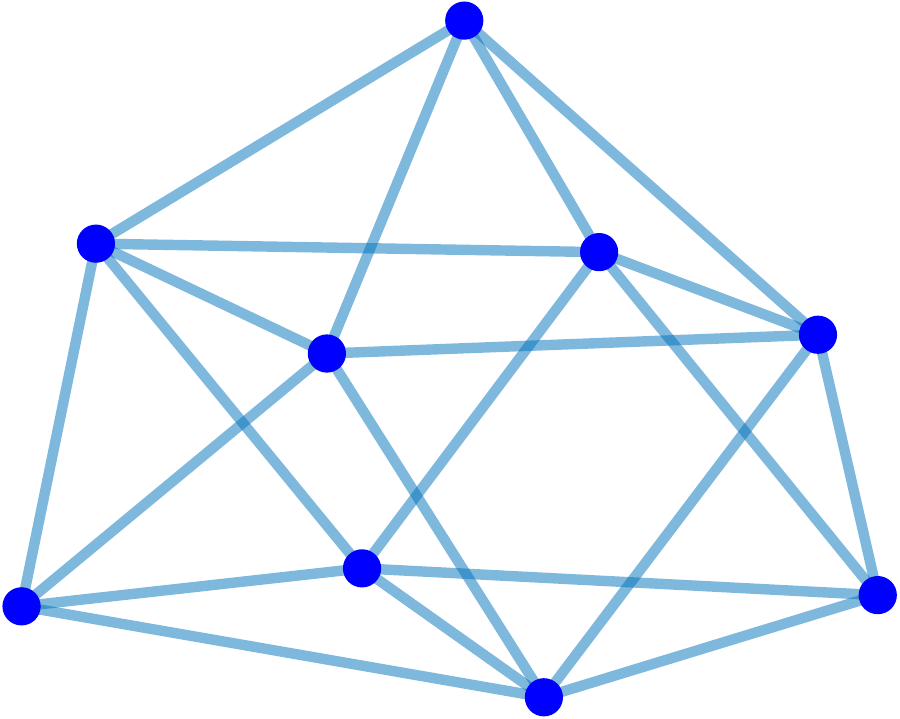}
    \quad
    \includegraphics[width=.25\textwidth]{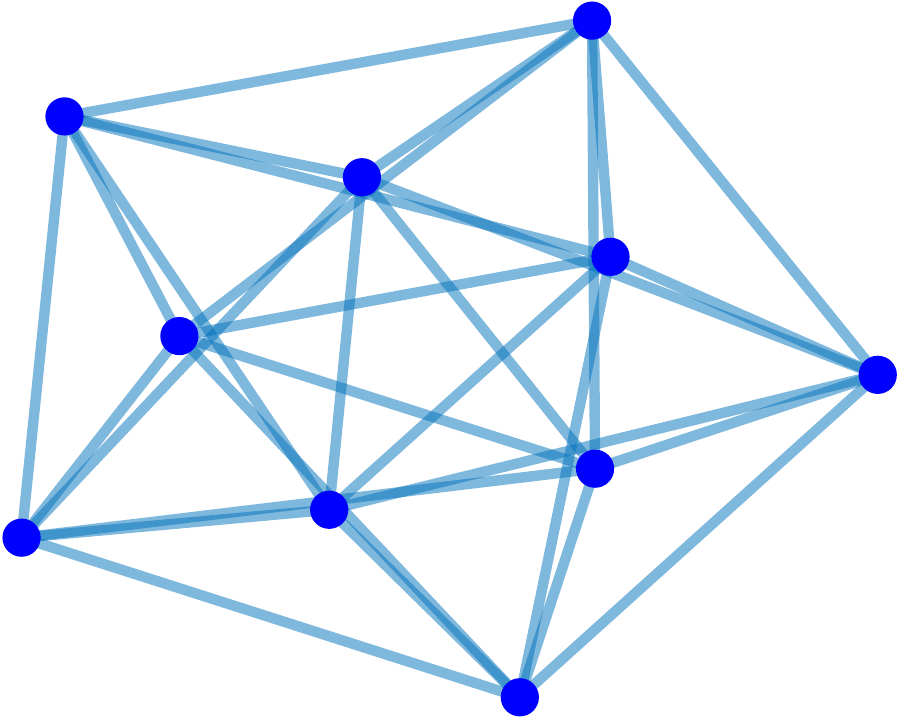}

{\small $m=8$ \hskip 3.5truecm $m=9$ \hskip 3.5truecm $m=10$}

    \caption{\small noncomplete graphs on $8\le m\le 10$  vertices maximizing $\Lambda^{pow}$ which is greater than the value $=2m -2$ attained by the complete graph $K_m$. For values of $\Lambda^{pow}$ see Table~\ref{tab-spektrumall}.}
    
    \label{fig-graf-maxLambdaPower-8-9-10}
\end{figure}

Interestingly enough, for the values of $m\le7$ the maximum value of $\Lambda^{pow}$ is achieved for the complete graph $K_m$ with the eigenvalues $\{m-1, -1, \dots, -1\}$ and the maximal value $\Lambda^{pow}=2m-2$. For $m=7$ there are exactly two graphs with the same maximal value $\Lambda^{pow}=12$. The noncomplete maximizing graph with eigenvalues $\{ 5,1,-1,-1,-1,-1,-2\}$ is shown in Fig.~\ref{fig-graf-maxLambdaPower-2-second}. Starting from the degree $m=8$ the maximal value of  $\Lambda^{pow}$  is attained for noncomplete graphs shown in Fig.~\ref{fig-graf-maxLambdaPower-8-9-10}.
In Fig.~\ref{fig-graf-minLambdaHLgap-5-6-7-8-9-10} we show graphs on $5\le m\le 10$  minimizing $\Lambda^{gap}$. Path graphs $P_m$ minimize $\Lambda^{gap}$ and $\Lambda^{ind}$ for $m=2,3,4$ (see Table~\ref{tab-spektrumall}). In Fig.~\ref{fig-graf-minLambdaHLindex-6-7-9-10} we show graphs on $m=6,7,9,10$  minimizing $\Lambda^{ind}$. For $m=5,8$ the minimizing graphs are the same as those for $\Lambda^{gap}$ shown in Fig.~\ref{fig-graf-minLambdaHLindex-6-7-9-10} (see Table~\ref{tab-spektrumall}). 

\begin{figure}
    \centering
    
    \includegraphics[width=.25\textwidth]{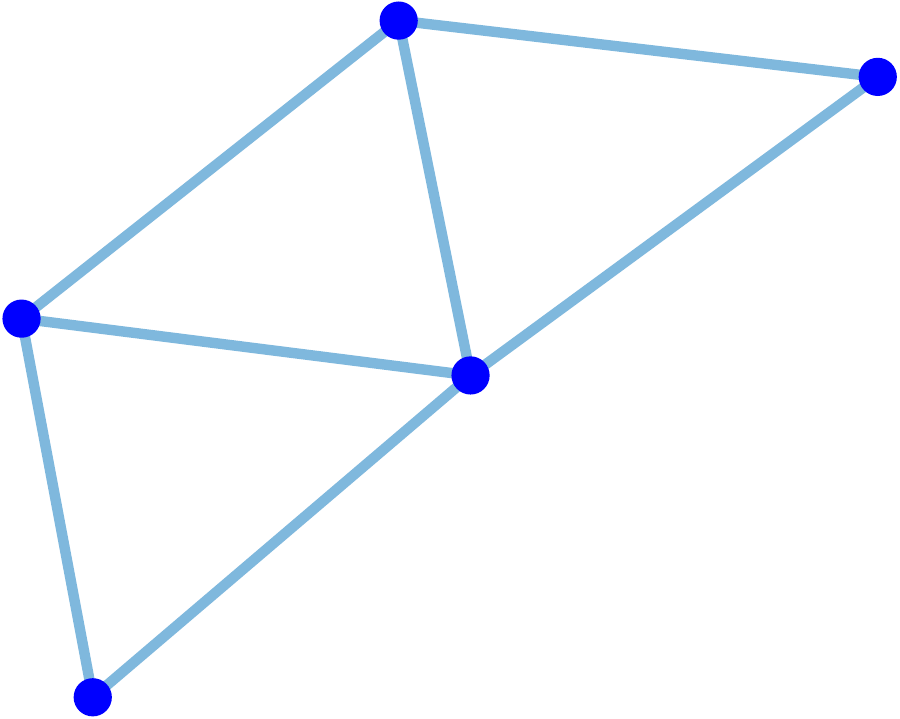}
    \quad
    \includegraphics[width=.25\textwidth]{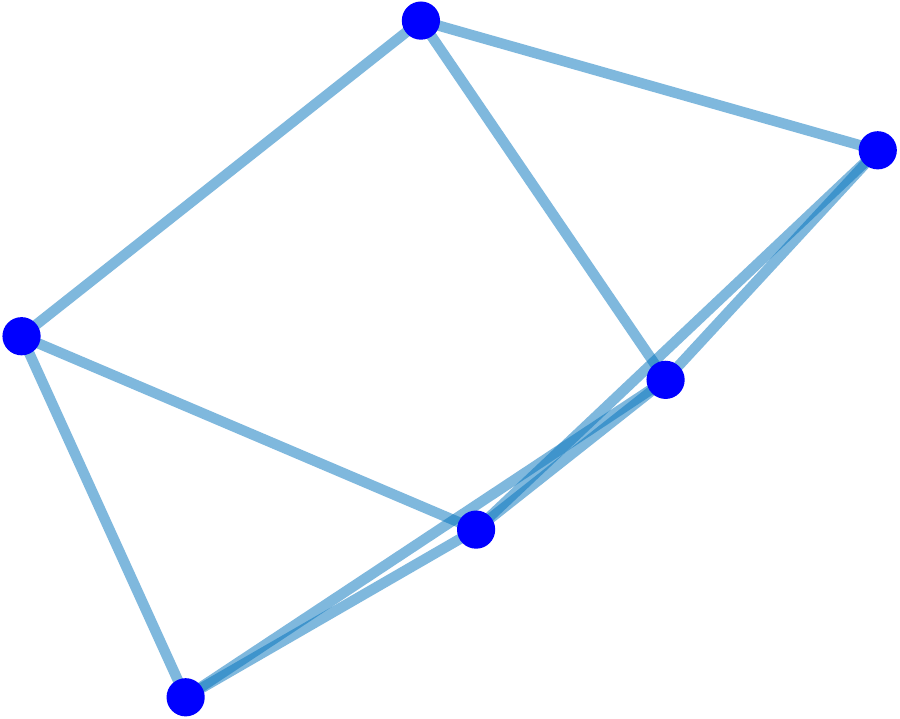}
    \quad
    \includegraphics[width=.25\textwidth]{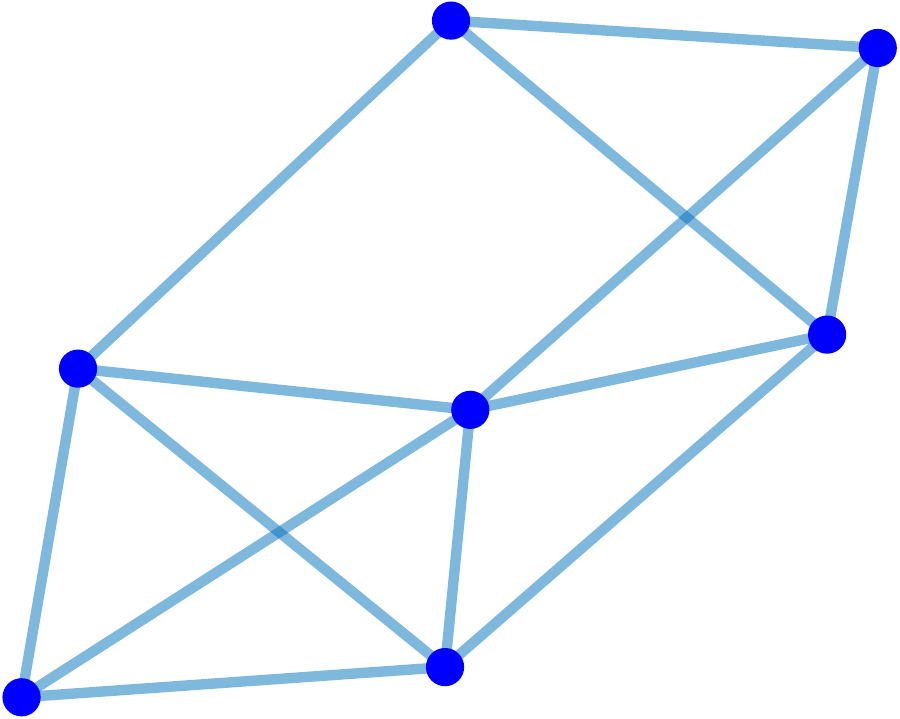}

{\small $m=5$ \hskip 3.5truecm $m=6$ \hskip 3.5truecm $m=7$}

\bigskip

\includegraphics[width=.25\textwidth]{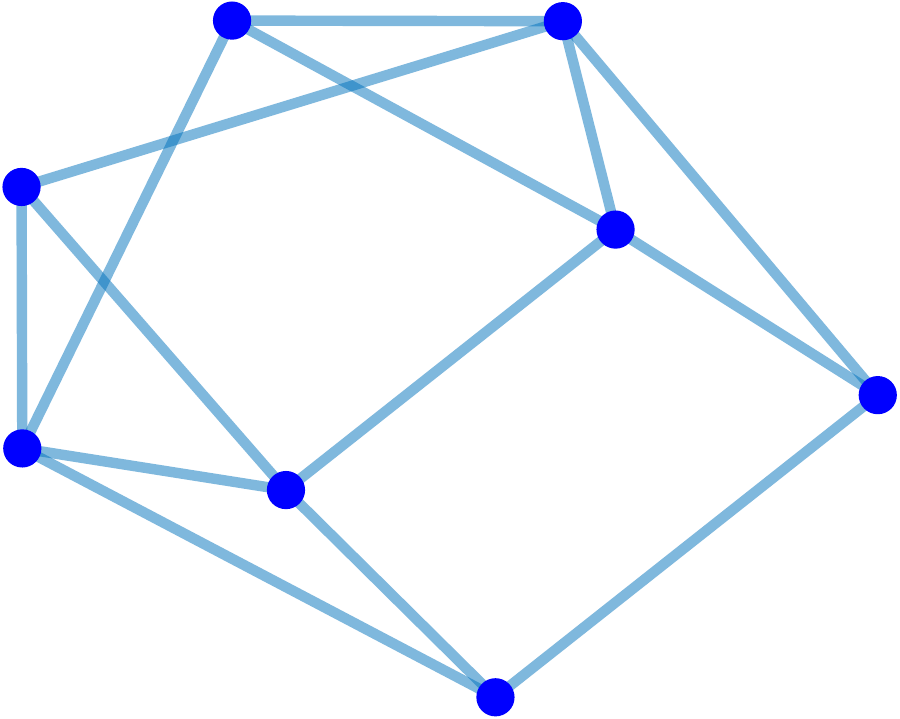}
    \quad
    \includegraphics[width=.25\textwidth]{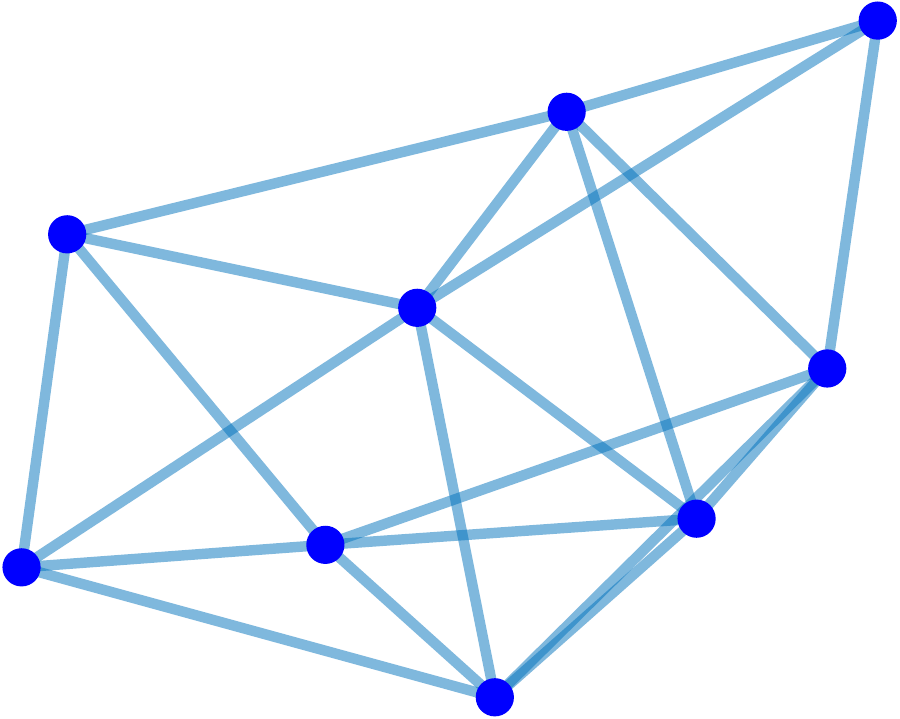}
    \quad
    \includegraphics[width=.25\textwidth]{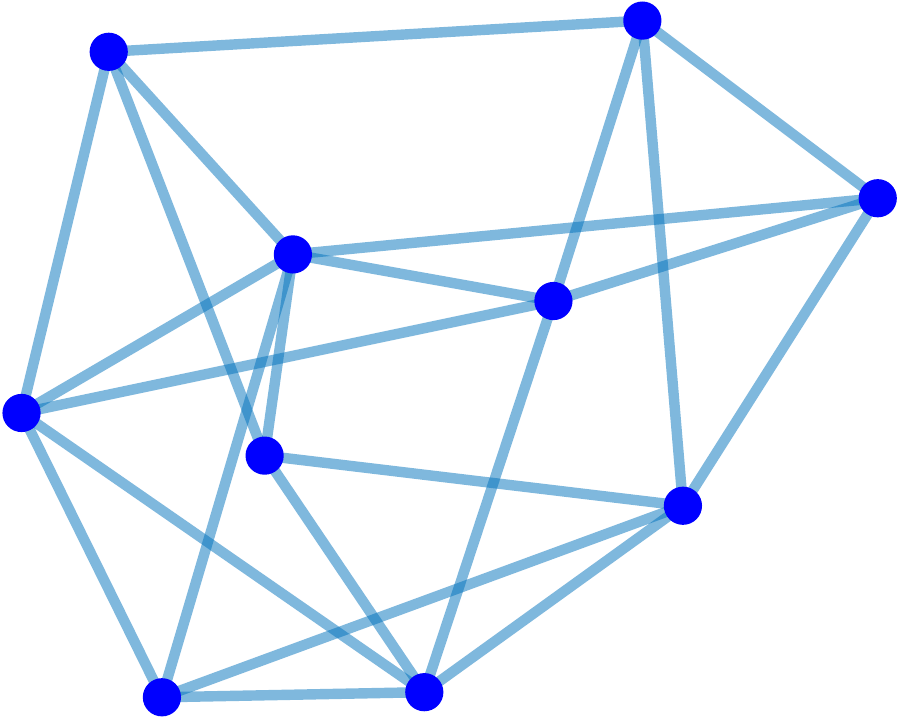}
    \quad

{\small $m=8$ \hskip 3.5truecm $m=9$ \hskip 3.5truecm $m=10$}

    \caption{\small Graphs on $5\le m\le 10$  minimizing $\Lambda^{gap}$. For values of $\Lambda^{pow}$ see Table~\ref{tab-spektrumall}.}
    
    \label{fig-graf-minLambdaHLgap-5-6-7-8-9-10}
\end{figure}

\begin{remark}
According to Caporossi {et al.} \cite[Theorem 2]{Cap}, for a general simple connected graph $G_A$ we have $\Lambda^{pow}(G_A) \ge 2\sqrt{m-1}$. The unique minimal value of $\Lambda^{pow}=2\sqrt{m-1}$ is attained by the star graph $S_m\equiv K_{m-1,1}$. For related results, we refer to Stanic \cite[(2.11), p. 33]{Sta} and McClelland \cite{McCle}.

\end{remark}

\begin{figure}
    \centering
    
 \hglue -1truecm   \includegraphics[width=.24\textwidth]{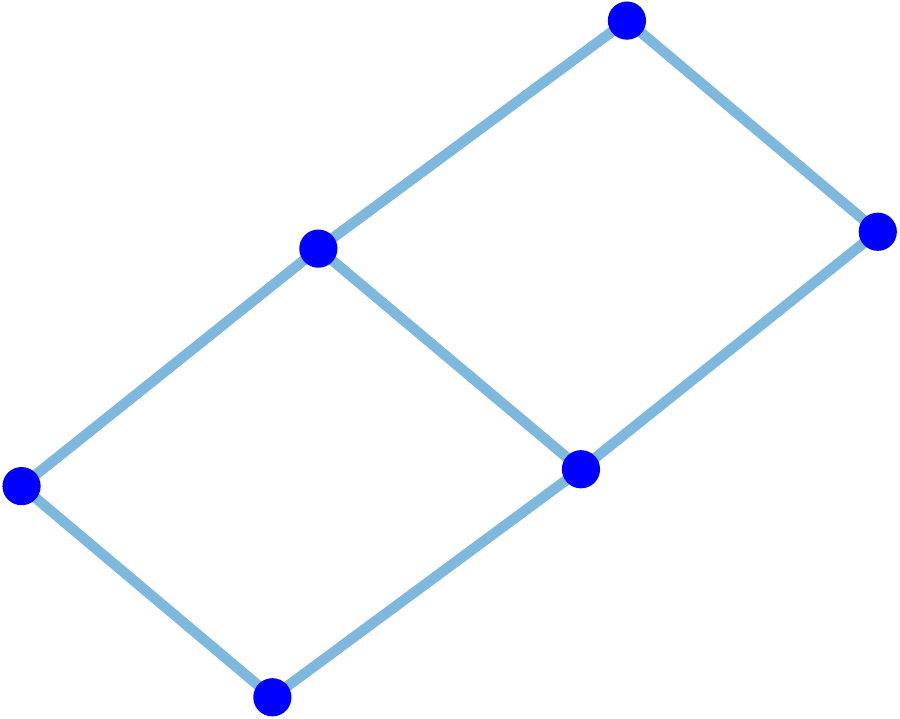}
    \includegraphics[width=.24\textwidth]{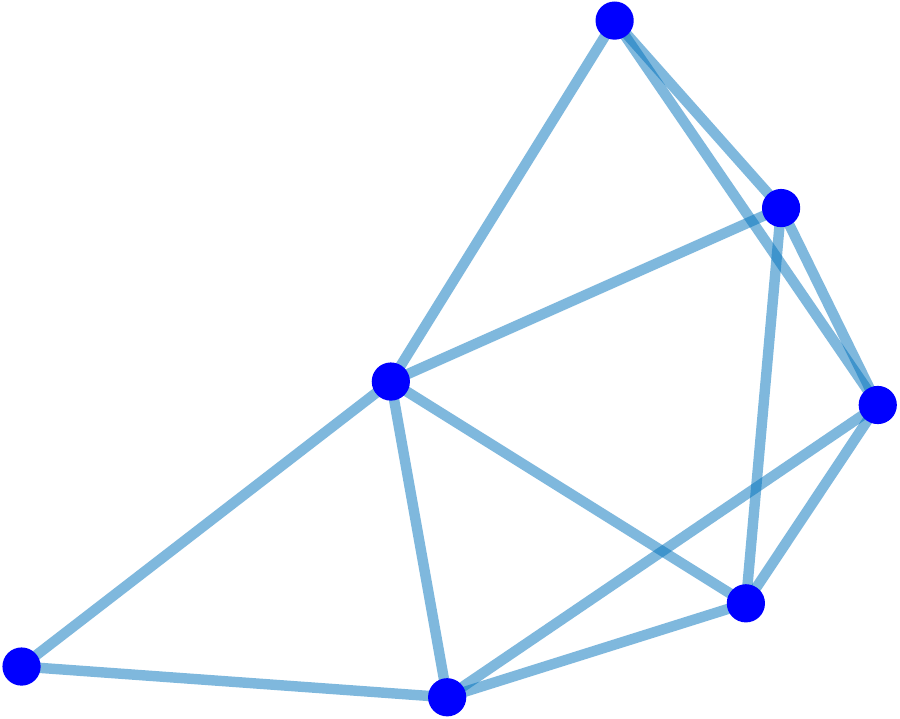}
    \includegraphics[width=.25\textwidth]{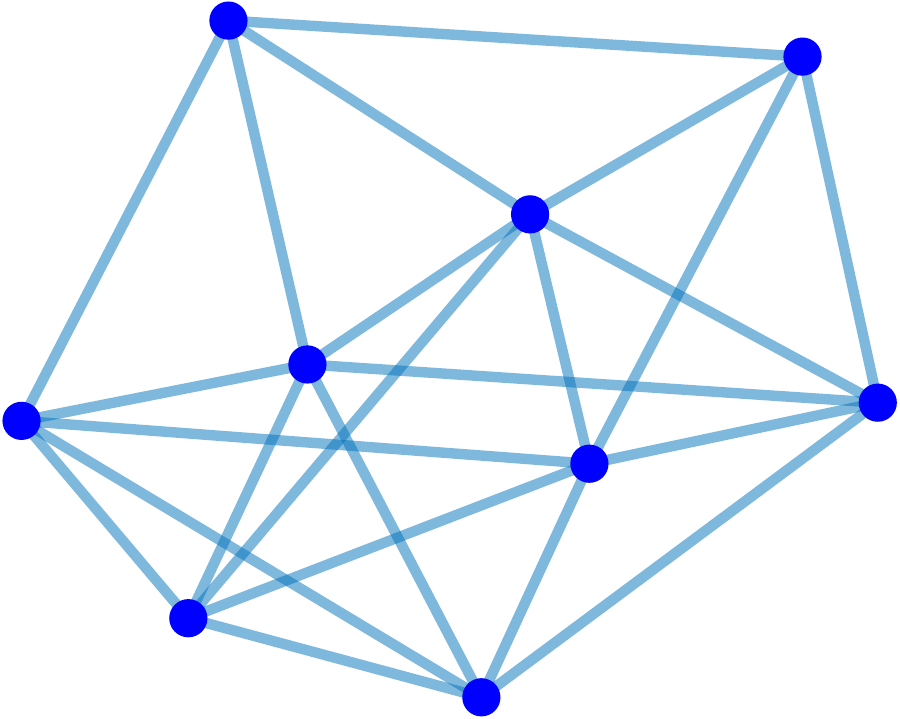}
    \includegraphics[width=.25\textwidth]{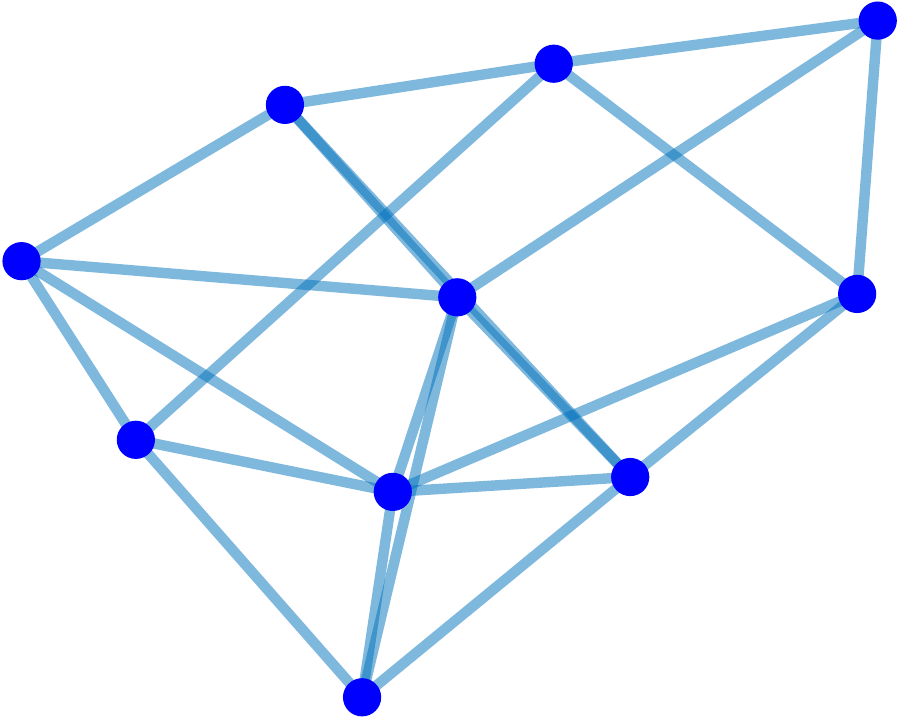}

{\small $m=6$ \hskip 2.5truecm $m=7$ \hskip 2.5truecm  $m=9$ \hskip 2.5truecm $m=10$}

\bigskip

    \caption{\small Graphs on $5\le m\le 10$  minimizing $\Lambda^{ind}$. For values of $\Lambda^{pow}$ see Table~\ref{tab-spektrumall}.}
    
    \label{fig-graf-minLambdaHLindex-6-7-9-10}
\end{figure}

\section{Conclusions}
In this paper we analyzed the spectral properties of all simple connected graphs. We focus our attention to the class of graphs which are complete multipartite graphs. We also present results on density of spectral gap indices and its nonpersistency with respect to small perturbations of the underlying graph. We also analyzed the spectral properties of graphs different from those of complete multipartite graphs. We presented statistical and numerical analysis of the  indices $\Lambda^{gap}, \Lambda^{ind}$, and $\Lambda^{pow}$ of graphs of order $m\le 10$. 

\section*{Acknowledgments}
Support of the Slovak Research and Development Agency under the projects APVV-19-0308 (SP, JS), and APVV-20-0311 (DS) is kindly acknowledged.


\begin{thebibliography}{9}


\bibitem{Aihara1999JCP}
\newblock J. I. Aihara,
\newblock Reduced HOMO-LUMO Gap as an Index of Kinetic Stability for Polycyclic Aromatic Hydrocarbons,
\newblock {J. Phys. Chem. A}, \textbf{103} (1999), 7487--7495.

\bibitem{Aihara1999TCH}
\newblock J. I. Aihara,
\newblock Weighted HOMO-LUMO energy separation as an index of kinetic stability for fullerenes,
\newblock {Theor. Chem. Acta}, \textbf{102} (1999), 134--138.


\bibitem{Bacalis2009}
\newblock N. C. Bacalis and A. D. Zdetsis,
\newblock Properties of hydrogen terminated silicon nanocrystals via a transferable tight-binding Hamiltonian, based on ab-initio results,
\newblock {J. Math. Chem.}, \textbf{26} (2009), 962--970.


\bibitem{Bapat} R. B. Bapat, and E. Ghorbani, Inverses of triangular matrices and bipartite graphs, Linear Algebra and its Applications 447 (2014), 68-73.

\bibitem{Bapat2010}
R. B. Bapat, Graphs and matrices. Universitext. Springer, London; Hindustan Book Agency, New Delhi, 2010, 171 pp.

\bibitem{Bender1990}
\newblock  E. A. Bender, E. R. Canfield, and B. D. McKay, 
\newblock The asymptotic number of labeled connected graphs with a given number of vertices and edges, 
\newblock {Random Structures and Algorithms}, \textbf{1} (1990),  127-169.





\bibitem{Ben}
\newblock A. Ben-Israel, T.N.E. Greville, Generalized Inverses: Theory and Applications, CMS Books Math., Springer, 2003.

\bibitem{Brouwer2012} Brouwer, A., Haemers, W., Spectra of graphs. Universitext. Springer, New York, 2012. xiv+250 pp. ISBN: 978-1-4614-1938-9 



\bibitem{Cap}
G. Caporossi, D. Cvetkovi\'c , I. Gutman, P. Hansen, 
Variable neighborhood search for extremal graphs.2. Finding graphs with extremal energy, 
J. Chem. Inf. Comput. Sci. 39 (1999) 984–996.

\bibitem{Con} G. Constantine, Lower bounds on the spectra of symmetric matrices with non-negative entries,
Linear Algebra Appl. 65 (1965), 171--178.  


\bibitem{Cvetkovic1988} D. Cvetkovi\'c, M. Doob, and H. Sachs, \textit{Spectra of graphs - Theory and application}, Deutscher Verlag der Wissenchaften, Berlin, 1980; Academic Press, New York, 1980.


\bibitem{Cvetkovic2004} 
\newblock D. Cvetkovi\'c, P. Hansen and V. Kova\v{c}evi\v{c}-Vu\v{c}i\v{c},
\newblock On some interconnections between combinatorial optimization and extreme graph theory,
\newblock {Yugoslav Journal of Operations Research}, \textbf{14} (2004), 147--154.


\bibitem{Cvetkovic1995} D. Cvetkovi\'c, and S. Simi\'c, 
The second largest eigenvalue of a graph (a survey). Filomat \textbf{9} (1995), 449-472.

\bibitem{CvDS} D. Cvetkovi\'{c}, M. Doob and H. Sachs, Spectra of graphs - Theory and application, 3rd Ed.,
Heidelberg-Leipzig, 1995.

\bibitem{Del} C. Delorme, Eigenvalues of complete multipartite graphs, Discrete Math. 312 (2012), 2532--2535.

\bibitem{ElMc} N. D. Elkies and C. T. McMullen. Gaps in $\sqrt{n}$ mod 1 and ergodic theory, Duke Math. J.  123 (2004) 1, 95--139.

\bibitem{EsHa} F. Esser and F. Harary, On the spectrum of a complete multipartite graph, Europ. J. Combin. 1 (1980), 211--218.

\bibitem{FoPi} P. Fowler and T. Pisanski, HOMO-LUMO maps for chemical graphs,  MATCH 64 (2010), 373--390. 

\bibitem{Fowler2001}
\newblock P. W. Fowler, P. Hansen, G. Caporosi and A. Soncini,
\newblock  {Polyenes with maximum HOMO-LUMO gap},
\newblock {Chemical Physics Letters}, \textbf{342} (2001), 105--112.

\bibitem{Fowler2010}
\newblock P. V. Fowler,
\newblock  {HOMO-LUMO maps for chemical graphs},
\newblock {MATCH Commun. Math. Comput. Chem.}, \textbf{64} (2010), 373--390.

\bibitem{Godsil1985} C. D. Godsil, Inverses of Trees, Combinatorica 5 (1985), 33--39.

\bibitem{Gutman1979}
\newblock I. Gutman and D.H. Rouvray,
\newblock  {An Aproximate TopologicaI Formula for the HOMO-LUMO Separation in Alternant Hydrocarboons},
\newblock {Chemical-Physic Letters}, \textbf{72} (1979), 384--388.

\bibitem{Har} F. Harary, and H. Minc, Which non-negative matrices are self-inverse? Math. Mag. (Math. Assoc. of America) 49 (1976) 2, 91--92.

\bibitem{Hon} Y. Hong, Bound of eigenvalues of a graph, Acta Math. Appl. Sinica 432 (1988), 165--168.

\bibitem{Huckel1931}
\newblock E. H\"uckel,
\newblock Quantentheoretische Beitr\"age zum Benzolproblem,
\newblock {Zeitschrift f\"ur Physik}, \textbf{30} (1931), 204--286.

\bibitem{Jaklic2012}
\newblock G. Jakli\'{c},
\newblock  {HL-index of a graph},
\newblock {Ars Mathematica Contemporanea}, \textbf{5} (2012), 99--105.


\bibitem{KirklandAkb2007} S. J. Kirkland, and S. Akbari, On unimodular graphs, Linear Algebra and its Applications 421 (2007), 3--15.

\bibitem{KirklandTif2009} S. J. Kirkland, and R. M. Tifenbach, Directed intervals and the dual of a graph,
Linear Algebra and its Applications 431 (2009), 792--807.


\bibitem{Lin2016}
\newblock Lin Chen and Jinfeng Liu,
\newblock  {extreme values of matching energies of one class of graphs},
\newblock {Applied Mathematics and Computation}, \textbf{273} (2016), 976--992.


\bibitem{Li2013}
\newblock Xueliang Li, Yiyang Li, Yongtang Shi and I. Gutman,
\newblock  {Note on the HOMO-LUMO index of graphs},
\newblock {MATCH Commun. Math. Comput. Chem.}, \textbf{70} (2013), 85--96.


\bibitem{LLSG} X. Li, Y. Li, Y. Shi and I. Gutman, Note on the HOMO-LUMO index of graphs, MATCH 70 (2013), 85--96


\bibitem{McCle}
B. J. McClelland,  Properties of the latent roots of a matrix: The estimation of $\pi$-electron energies. 
J. Chem. Phys., 54 (1971), 640–643.


\bibitem{McKay}
\newblock B. McKay,
\newblock Combinatorial Data, 
\newblock Available online Nov/2022 \\ http://users.cecs.anu.edu.au/~bdm/data/graphs.html


\bibitem{Mohar2013} 
\newblock B. Mohar,
\newblock  {Median eigenvalues of bipartite planar graphs},
\newblock {MATCH Commun. Math. Comput. Chem.} \textbf{70} (2013), 79--84.

\bibitem{Mohar2015} 
\newblock M. Mohar,
\newblock  {Median eigenvalues and the HOMO-LUMO index of graphs},
\newblock {Journal of Combinatorial Theory, Series B}, \textbf{112} (2015), 78--92.


\bibitem{Pavlikova2016}
\newblock S. Pavl\'{\i}kov\'{a}, and D. \v{S}ev\v{c}ovi\v{c},
\newblock On a construction of integrally invertible graphs and their spectral properties,
\newblock {Linear Algebra and its Applications}, \textbf{532} (2017), 512--533.

\bibitem{Pavlikova2022-LAA}
\newblock S. Pavl\'{\i}kov\'{a}, and D. \v{S}ev\v{c}ovi\v{c},
\newblock  On the Moore-Penrose pseudo-inversion of block symmetric matrices and its application in the graph theory, 
\newblock {Linear Algebra and its Applications}, \textbf{673} (2023), 280-303.

\bibitem{Pow} D. L. Powers, Graph partitioning by eigenvectors, Linear Algebra Appl. 101 (1988), 121--133.

\bibitem{Powers1987} D. Powers: 
Structure of a matrix according to its second eigenvector. 
Current trends in matrix theory, Proc. 3rd Conf., Auburn/Ala. 1986, 261-265 (1987). 

\bibitem{Powers1989} D. Powers: 
Bounds on graph eigenvalues.
Linear Algebra Appl. 117, 1-6 (1989). 


\bibitem{Sta} 
Z. Stani\v{c}, 
Inequalities for Graph Eigenvalues, LMS Lect. Notes Ser. 423, Cambridge Univ. Press, 2015.

\bibitem{Streitwieser1961}
\newblock Streitwieser, A., 
\newblock Molecular orbital theory for organic chemists, 
\newblock John Willey \& Sons, New York-London, 1961. 



\bibitem{Ye} D. Ye, Y. Yang, B. Manda, and D. J. Klein, Graph invertibility and median eigenvalues,
Linear Algebra and its Applications, 513(15) (2017), 304--323.

\bibitem{Zas} 
\newblock T. Zaslavsky, 
\newblock Signed graphs, 
\newblock {Discrete Applied Math.}, \textbf{4} (1982), 47--74.



\bibitem{Zhang2002}
\newblock F. Zhang and Z. Chen,
\newblock  {Ordering graphs with small index and its application},
\newblock {Discrete Applied Mathematics}, \textbf{121} (2002), 295--306.


\end{thebibliography}
\end{document}